\documentclass[oneside,english,reqno]{amsart}
\usepackage{algorithm}
\usepackage{algorithmic}
\usepackage{multirow}
\usepackage{subfigure}
\usepackage{enumerate}

\usepackage[T1]{fontenc}
\usepackage[latin9]{inputenc}
\usepackage{babel}
\usepackage{amsthm}
\usepackage{amssymb}
\usepackage{graphicx}
\usepackage[unicode=true,pdfusetitle,
bookmarks=true,bookmarksnumbered=false,bookmarksopen=false,
breaklinks=false,hidelinks,backref=false,colorlinks=false]
{hyperref}

\makeatletter
\numberwithin{equation}{section}
\numberwithin{figure}{section}

\makeatother
\newcommand{\hsubset}{\ensuremath{C}}
\newcommand{\solset}{\ensuremath{Z}}
\newcommand{\secondhsubset}{\ensuremath{H}}
\newcommand{\prox}{\ensuremath{\text{Prox}}}

\makeatletter
\renewcommand{\ALG@name}{Iterative Scheme}
\makeatother

\newtheorem{proposition}{Proposition}
\newtheorem{theorem}{Theorem}
\newtheorem{lemma}{Lemma}
\newtheorem{remark}{Remark}
\newtheorem{corollary}{Corollary}

\begin{document}
	%	\title{Inertial proximal best approximation primal-dual algorithm}
	%
	%
	%	
	
	\title{Proximal primal-dual best approximation  algorithm with memory}
	

	\author{
		Ewa M. Bednarczuk$^1$ 
	}
	\thanks{$^1$ System Research Institute, Warsaw University of Technology, 		\href{mailto:e.bednarczuk@mini.pw.edu.pl}{e.bednarczuk@mini.pw.edu.pl}}
	\author{
		Anna Jezierska$^2$
	}
	\thanks{$^2$ System Research Institute, Gdansk University of Technology, \href{mailto:Anna.Jezierska@ibspan.waw.pl}{Anna.Jezierska@ibspan.waw.pl}
		\url{http://www.ibspan.waw.pl/\~jeziersk/}		
	}
	\author{
		Krzysztof E. Rutkowski$^3$ 
	}
	\thanks{$^3$ Warsaw University of Technology, 		 \href{mailto:k.rutkowski@mini.pw.edu.pl}{k.rutkowski@mini.pw.edu.pl}}
	\maketitle
	\tableofcontents{}
	
	\begin{abstract}
	We propose a new modified primal-dual proximal best approximation method for solving convex not necessarily differentiable optimization problems. The novelty of the method relies on introducing memory 
	by taking into account iterates computed in previous steps in the formulas defining current iterate.
	To this end we consider projections onto  
	intersections of halfspaces generated on the basis of the current as well as the previous iterates. To calculate these projections we are using recently obtained closed-form expressions for projectors onto polyhedral sets. 
	The resulting  algorithm with memory inherits strong convergence properties of the original best approximation proximal primal-dual algorithm.
	Additionally, we  compare our algorithm with the original (non-inertial) one  with the help of  the so called \textit{attraction property} defined below.
	Extensive numerical experimental results on image reconstruction problems
	illustrate the advantages of including memory into the original algorithm. 

	\end{abstract}

	\section{Introduction}
	Motivated by problems arising in the field of inverse problems, signal processing, computer vision and machine learning, there
	has been an increasing interest in primal-dual methods \cite{Bot2013,playing_with_duality,Li2016}.
	Over the last years, substantial progress has been made. Among others, the recent advances concern block algorithms \cite{Chouzenoux_2016_block_coordinate_fb,a_simplified_form_of_block}, asynchronous methods \cite{asynchronous_block,ARock}, generalizations of projection
	algorithms \cite{Combettes_2015_best_bergman,He2015} and introduction of memory effect. 
	
	While versions with memory of several proximal primal-dual algorithms already exist \cite{inertial_proximal_admm,a_general_inertial_proximal_point_algorithm_for_mixed_variational,an_inertial_forward_backward_monotone_inclusions,Ochs_2015_ipiasco,a_parallel_inertial_proximal_optimization_method}, in this paper we propose a new way of
	introducing memory effect in projection algorithms by studying algorithm \cite{Alotaibi_2015_best}. 
	We consider the following convex optimization problem
	\begin{equation}
		\label{prob_main}
		\min\limits_{p\in H}\ f(p)+g(Lp),
	\end{equation}
	where $H$ and $G$ are two real Hilbert spaces, $f:\ H \rightarrow \mathbb{R}\cup \{+\infty\}$, $g:\ G\rightarrow\mathbb{R}\cup \{+\infty\}$ are proper convex lower semi-continuous functions and $L:\ H\rightarrow G$ is a bounded linear operator. 
	Under suitable regularity conditions  problem \eqref{prob_main}
	is equivalent to   the problem of finding $p \in H$ such that
	\begin{equation}\label{inclusion:primal}
		0 \in \partial\, f (p)+L^{*}\partial\, g(Lp),
	\end{equation}
	where $\partial(\cdot) $ denotes the subdifferential set-valued 
	operator. Problem \eqref{inclusion:primal} is of the form
	\begin{equation}\tag{P}\label{inclusion:primal-max}
		0 \in A(p)+L^{*}B(Lp),
	\end{equation}
	where $A:\ H\rightrightarrows H$ and $B:\ G\rightrightarrows G$ are maximally monotone set-valued operators.

	Different approaches to solve
	\eqref{inclusion:primal-max} have been proposed e.g. in
	\cite{a_parallel_splitting_method,general_projective_splitting_methods,a_stochastic_inertial_forward_backward}. 
	In particular, primal-dual approaches to solve \eqref{prob_main} may lead to formulations which can be represented as in \eqref{inclusion:primal-max}, see e.g. 
	\cite{Alotaibi_2015_best,Alotaibi_2014_sucessive_fejer_approximation,inertial_radu,Bot_2015_convergence_rate_primal_dual,Douglas_Rachford_type,on_the_ergdoic_covergence_rates_Chambolle,Combettes2012} and the references therein. Recently, the primal-dual approach has been applied in \cite{a_new_splitting} to a more general form of (\ref{inclusion:primal-max}) involving the sum of 
	two maximally monotone operators and a monotone operator. 
	The case when $A$ is maximal monotone and $B$ is strongly monotone was considered in \cite{a_stochastic_inertial_forward_backward}. The overview of primal-dual approaches to solve \eqref{inclusion:primal-max} has been recently proposed in \cite{playing_with_duality}. 
	
	Some algorithms to solve \eqref{prob_main} 
	which rely on including $x_{n-1}$ into the definition of $x_{n+1}$ were proposed in \cite{weak_convergence_of_relaxed_and_inertial_hybrid,Alvarez2001,inertial_radu,on_the_ergdoic_covergence_rates_Chambolle,Johnstone2017,numerical_approach_to_a_stationary_solution,inertial_iteratice_process_for_fixed_points,Mainge_2014_inertial_km_alg,regularized_and_inertial_alg_for_common_fixed_points_nonlinear,a_hybrid_inertial_projection_proximal,an_approximate_inertial_proximal_method_enlargement,Ochs_2015_ipiasco,Ochs_2014_ipiano}. They are mostly based on discretizations of the second order differential system related to the problem (\ref{inclusion:primal}).
	This system, called \textit{heavy ball with friction}, is exploited in order to accelerate convergence. Indeed, the introduction of the  inertial term was shown to improve  the speed of convergence significantly \cite{Johnstone2017,numerical_approach_to_a_stationary_solution}.

	In \cite{a_parallel_inertial_proximal_optimization_method}  Pesquet and Pustelnik proposed a primal method to solve \eqref{prob_main} with inertial effect introduced through inertia parameters. 
	The method explores information from more than one previous steps and allows finding zeros of
	the sum of an arbitrary finite number of maximally monotone operators (see also \cite{general_projective_splitting_methods}).

	For monotone inclusion problems \eqref{inclusion:primal-max} inertial proximal algorithms and fixed-points iterations have been  proposed in \cite{weak_convergence_of_relaxed_and_inertial_hybrid,Alvarez2001,an_inertial_alternating_direction,inertial_radu,inertial_douglas_rachford,an_inertial_forward_backward_monotone_inclusions,an_approximate_inertial_proximal_method_enlargement,Moudafi_2003_proximalconvergence,a_stochastic_inertial_forward_backward}.

	In the present paper we propose a new projection algorithm with memory. 
	We introduce a memory effect into projection algorithms %through 
	by relying on successive projections onto   polyhedral sets
	constructed with the help of
	halfspaces 
	originating from current and previous iterates. To the best of our knowledge this way of introducing memory has not been  considered yet.

	By applying to problem \eqref{inclusion:primal-max} the generalized Fenchel-Rockafellar duality framework \cite[Corollary 2.12]{pennanen1999dualization} (see also Corollary 2.4 of \cite{pennanen_dualization}) we obtain the dual inclusion problem which amounts  to finding $v^*\in G$ such that
	\begin{equation}\tag{D}\label{inclusion:dual}
		0 \in -LA^{-1}(-Lv^*)+B^{-1}v^*.
	\end{equation}
	By \cite[Corollary 2.12]{pennanen1999dualization}, a point $p\in H$ solves (\ref{inclusion:primal-max}) if and only if $v^*\in G$ solves (\ref{inclusion:dual}) and $(p,v^*)\in Z$, where
	\begin{equation}\label{set:Z}
		Z:=\{ (p,v^*)\in H\times G \ |\ -L^* v^* \in Ap \quad \text{and}\quad Lp \in B^{-1}v^*  \}.
	\end{equation}
	In the case when $L=Id$ and $H=G$, the set $Z$ reduces to the \textit{extended solution set} $S_{e}(A,B)$ as defined in \cite{A_family_of_projective_splitting_methods}.
	The set $Z$ is a closed convex subset of $H\times G$ (see e.g. \cite[Proposition 23.39]{Bauschke_2011_convex_analysis}).
	
	The Fenchel-Rockafellar dual problem of \eqref{prob_main}
	takes the form (see \cite{pennanen_dualization})
	\begin{equation}\label{optimization_funD}
		\min_{v^*\in G}\ f^*(-L^*v^*)+g^*(v^*),
	\end{equation}
	where $f^*$ denotes the conjugate function \cite{Rockafellar_1970_convex_analysis}. In this case set $Z$ is of the form
	\begin{equation}
		Z = \{(p,v^*)\in H\times G \ |\ -Lv^* \in \partial f(x)\ \text{and}\ v^*\in \partial g(Lx)  \}.
	\end{equation}

	\subsection{Projection methods}

	The idea 
	of finding a point in $Z$ is based on the fact that
	\begin{equation*}
		Z \subset \{ (p,v^*)\in H \times G \ |\ \varphi(p,v^*) \leq 0  \}:=H_\varphi,
	\end{equation*}
	where $\varphi(p,v^*):=\langle p - a \ |\ a^*+L^*v^* \rangle + \langle b^*-v^* \ |\ Lp - b \rangle$, $(a,a^*)\in \text{graph}A$, $(b,b^*)\in \text{graph}B$.
	This  suggests the following 
	iterative scheme for finding a point in $Z$  based on projections onto $H_\varphi$:
	for any $(p_0,v_0^*)\in H\times G$ and relaxation parameters $\lambda_n\in (0,2)$, $n\in \mathbb{N}$ let
	\begin{equation}\label{projection-method}
		(p_{n+1},v_{n+1}^*):=(p_n,v_n)+\lambda_n(P_{H_n}(p_n,v_n^*)-(p_n,v_n^*))  ,
	\end{equation}
	where $H_n:=\{ (p_n,v_n^*)\in H \times G \ |\ \varphi(p_n,v_n^*) \leq 0  \}$ with $\varphi_n$  defined for suitably chosen $(a_n,a_n^*)\in \text{graph}A$, $(b_n,b_n^*)\in \text{graph}B$ (\cite[Proposition 2.3]{Alotaibi_2014_sucessive_fejer_approximation}, see also \cite[Lemma 3]{A_family_of_projective_splitting_methods}) and $P_D(x)$ denoting the projection of $x$ onto the set $D$.
	For $L=Id$ this iteration scheme has been proposed by Eckstein and Svaiter \cite{A_family_of_projective_splitting_methods}  and its fundamental convergence properties has been investigated in  \cite[Proposition 1,  Proposition 2]{A_family_of_projective_splitting_methods}.

	Further convergence properties of \eqref{projection-method} have been investigated in \cite{Alotaibi_2014_sucessive_fejer_approximation} and  \cite[Theorem 2]{A_Note_on_the_Paper_by_Eckstein_and_Svaiter}. The sequence generated by \eqref{projection-method} is Fej{\'e}r monotone with respect to set $Z$ and, in general, only its weak convergence is guaranteed.
	
	Modifications of \eqref{projection-method} to force strong convergence have been proposed in \cite{Alotaibi_2015_best,strong_convergence_of_a_splitting_two,forcing_strong_2000,a_new_splitting,modified_extragradient_method}.  Recently, asynchronous block-iterative methods are proposed in \cite{asynchronous_block,a_simplified_form_of_block}.
	\subsection{The aim}
	In the present paper we propose a primal-dual projection algorithm with memory to solve (\ref{inclusion:primal-max})
	which relies on finding a point in the set $Z$ defined by \eqref{set:Z}. 
	The origin of our idea goes back to the algorithm of Haugazeau \cite[Corollary 29.8]{Bauschke_2011_convex_analysis}%\cite{Haugazeau}
	, who proposed an algorithm for finding the projection of $x_0\in H$ onto the intersection of a finite number of closed convex sets by using projections of $x_0$ onto intersections of two halfspaces. 
	These halfspaces are defined on the basis of the current iterate $x_n$ (see also \cite{strong_convergence_of_a_splitting_two,forcing_strong_2000,strong_convergence_of_a_splitting_projection}).  
	
	In our approach we take into account projections of $x_0$ onto intersections of three halfspaces which are defined on the basis of not only $x_n$ but also $x_{n-1}$. 
	
	The contribution of the paper is as follows.
	\begin{itemize}	
		\item We apply formulas for 
		projections onto intersections of three half-spaces in Hilbert spaces derived in \cite{explicit_formulas_KR}. We show that in the considered cases (Proposition \ref{prop:setdefinitions}) the complete enumeration is not required (Proposition \ref{projection_reduced}).
		\item We propose a number of iterative schemes with memory for solving  primal-dual 
		problems defined by \eqref{inclusion:primal-max} and \eqref{inclusion:dual}.
		\item We apply our iterative schemes to propose a proximal algorithm with memory to solve
		minimization problem defined by a  finite sum of convex functions.
		\item We provide convergence  comparison of the proposed algorithm with its non-memory version  in terms of \textit{attraction property} (Proposition \ref{prop:inertial_property}).
		\item We perform an experimental study aiming at comparing the best approximation algorithm proposed in \cite{Alotaibi_2015_best} and our algorithm. 
	\end{itemize}
	
	The organization of the paper is as follows.
	In section \ref{sec:proposed_approach} we  propose the underlying iterative schemes with memory  and we formulate basic convergence results. In section \ref{sec:choice} we provide several versions 
	of the iterative scheme with memory. One of the main ingredients is a  closed-form formula for projectors onto polyhedral sets introduced in \cite{explicit_formulas_KR}. In section \ref{sec:convergence_analysis} we perform the convergence  comparison of the proposed iterative schemes. In section \ref{sec:the_algorithm} we cast our general idea so as to be able to solve optimization problem of minimization of the sum of two convex, not necessarily differentiable functions. In section \ref{sec:experiments} we present the results of the numerical experiment.

	\section{The proposed approach}
	\label{sec:proposed_approach}
	
	In section \ref{subsec:succesive_fejer} we recall generic Fej{\'e}r Approximation Scheme for finding an element from the set $Z$ defined by \eqref{set:Z} and its basic properties. In section \ref{subsec:best_approx} we propose refinements of Fej{\'e}r Approximation Scheme which are based on the idea proposed 
	by Haugazeau \cite{Haugazeau}, see also 
	\cite[Corollary]{Bauschke_2011_convex_analysis}. 
	The crucial issue of the proposed refinements is to improve  convergence properties. 
	
	In the sequel, for any $x \in H\times G$ we write $x = \left(p,v^*\right)$, where $p\in H$ and $v^*\in G$.
	
	\subsection{Successive Fej{\'e}r Approximations iterative scheme}\label{subsec:succesive_fejer}
	Let $H$, $G$ be real Hilbert spaces and let $Z$ be defined by (\ref{set:Z}). Let $\{H_n\}_{n\in \mathbb{N}}\subset H\times G$, be a sequence of convex closed sets such that $Z\subset H_n,\ n\in \mathbb{N}$. The projections of any $x\in H$ onto $H_n$ are uniquely defined.
	
	\begin{algorithm}[H]
		\caption{Generic 
			%primal-dual 
			Fej{\'e}r Approximation Iterative Scheme 	\label{alg:generic_fejer}}
		\begin{algorithmic}
			\STATE{Choose an initial point $x_0\in H\times G$} 
			\STATE{Choose a sequence of parameters $\{\lambda_n\}_{n\ge 0}\in (0,2)$}
			\FOR{$n=0,1\dots$}
			%\STATE{Choose $H_n$ convex closed such that $Z\subset H_n$}
			\STATE{$x_{n+1}=x_n+\lambda_n(P_{H_n}(x_{n})-x_n)$}
			\ENDFOR
			\RETURN
		\end{algorithmic}
	\end{algorithm}
	\begin{theorem}\label{theorem:fejer_scheme}  (\cite[Proposition 3.1]{Alotaibi_2014_sucessive_fejer_approximation}, see also \cite{Combettes_2009_fejer_monotonicity})
		For any sequence generated by Iterative Scheme \ref{alg:generic_fejer} the following hold:
		\begin{enumerate}
			\item $\{x_n\}_{n\in \mathbb{N}}\subset H\times G$ is Fej{\'e}r monotone with respect to the set $Z$, i.e 
			\begin{displaymath}
				\forall_{n\in\mathbb{N}} \ \forall_{z\in Z}\ \|x_{n+1}-z\|\leq \|x_n-z\|,
			\end{displaymath}
			\item $\sum\limits_{n=0}^{+\infty} \lambda_n (2-\lambda_n) \|P_{H_n}(x_{n})-x_n\|^2< +\infty$,
			\item if
			\begin{displaymath}
				\forall x\in H\times G \
				\forall \{k_n\}_{n\in \mathbb{N}}\subset\mathbb{N} \quad x_{k_n} \rightharpoonup x \implies x \in Z,
			\end{displaymath}	 
			then $\{x_{n}\}_{n\in \mathbb{N}}$ converges weakly to a point in $Z$.
		\end{enumerate}
		
	\end{theorem}
	
	In \cite{Alotaibi_2014_sucessive_fejer_approximation} the sets $H_n$ appearing in Iterative Scheme \ref{alg:generic_fejer} are defined as closed halfspaces $H_{a_n,b_n^*}$, 
	\begin{align}\label{settings2}
		%\tag{*}
		\begin{aligned}
			H_{a_n,b_n^*} &:= \left\{x \in H\times G \mid \left\langle  x \mid s_{a_n,b_n^*}^*\right\rangle \leq \eta_{a_n,b_n^*}  \right\},\\
			s_{a_n,b_n^*}^*  &:= (a_n^* + L^*b_n^*, b_n - La_n) ,  \\
			\eta_{a_n,b_n^*} &:=\left\langle a_n \mid a_n^* \right\rangle + \left\langle b_n \mid b_n^* \right\rangle,
		\end{aligned}
	\end{align}
	with
	\begin{align*}
		\begin{aligned}%\label{def:Hn} \\
			& a_n := J_{\gamma_n A} (p_n - \gamma_n L^* v_n^*),\quad b_n := J_{\mu_n B} (L p_n + \mu_n v_n^*) ,  \\
			& a_n^* := \gamma_n^{-1}(p_n-a_n) - L^* v_n^* ,\quad b_n^* := \mu_n^{-1} (Lp_n-b_n) +  v_n^*,
		\end{aligned}
	\end{align*}
	where for any maximally monotone operator $D$ and constant $\xi>0$, $J_{\xi D}(x)=(Id+\xi D)^{-1}(x)$. Parameters $\mu_n,\gamma_n>0$ are suitable defined. 
	It easy to see $H_{\varphi_n}=H_{a_n,b_n^*}$, where $\varphi_n=\varphi(a_n,b_n^*)$. 
	
	For $H_n=H_{a_n,b_n^*}$ Theorem \ref{theorem:fejer_scheme} can be strengthened in following way.
	\begin{theorem}\cite[Proposition 3.5]{Alotaibi_2014_sucessive_fejer_approximation} For any sequence generated by Iterative Scheme \ref{alg:generic_fejer} with $H_n$ defined by $(\ref{settings2})$ the following hold:
		\begin{enumerate}
			\item  $\{x_{n}\}_{n\in \mathbb{N}}=\{(p_n,v_n^*)\}_{n\in \mathbb{N}}$ is Fej{\'e}r monotone with respect to the set $Z$,
			\item $\sum\limits_{n=0}^{+\infty} \|a_n^*+L^*b_n^*\|^2 < +\infty $ and $\sum\limits_{n=0}^{+\infty} \|La_n-b_n\|^2 < +\infty $,
			\item $\sum\limits_{n=0}^{+\infty} \|p_{n+1}-p_n\|^2 < +\infty $ and $\sum\limits_{n=0}^{+\infty} \|v_{n+1}^*-v_n^*\|^2 < +\infty $,
			\item  $\sum\limits_{n=0}^{+\infty} \|p_n-a_n\|^2 < +\infty $ and $\sum\limits_{n=0}^{+\infty} \|v_n^*-b_n^*\|^2 < +\infty $,
			\item $\{x_{n}\}_{n\in \mathbb{N}}$ converges weakly to a point in $Z$.
		\end{enumerate}
	\end{theorem}
	
	\subsection{Best approximation iterative schemes} \label{subsec:best_approx}
	Here we study iterative best approximation schemes in the form of Iterative Scheme \ref{alg:generic_best}. For any $x,y\in H\times G$ we define
	\begin{displaymath}
		H(x,y):=\{h\in H\times G\ |\ \langle h-y \ |\ x-y\rangle\le 0   \}.
	\end{displaymath}
	
	As previously, let $\{H_n\}_{n\in \mathbb{N}}\subset H\times G$ be a sequence of closed convex sets, $Z\subset H_n$ for $n\in \mathbb{N}$.  
	
	\begin{algorithm}[H]
		\caption{Generic primal-dual best approximation iterative scheme \label{alg:generic_best}}
		\begin{algorithmic}
			\STATE{Choose an initial point $x_0=(p_0,v_0^*)\in H\times G$} 
			\STATE{Choose a sequence of parameters $\{\lambda_n\}_{n\ge 0}\in (0,1]$}
			\FOR{$n=0,1\dots$}
			\STATE{\bf{Fej{\'e}rian step}}
			%\STATE{Choose $\secondhsubset_n$ such that $Z\subset \secondhsubset_n$} 
			\STATE{$x_{n+1/2}=x_n+\lambda_n(P_{\secondhsubset_n}(x_{n})-x_n)$}
			\STATE{Let $\hsubset_n$ be a closed convex set such that $Z\subset C_n\subset H(x_n,x_{n+1/2})$.}
			\STATE{\bf{Haugazeau step}}
			%\STATE{Choose $\hsubset_n$ such that $Z\subset \hsubset_n \subset \secondhsubset_n$} 
			\STATE{$x_{n+1}=P_{H(x_0, x_n) \cap \hsubset_n}(x_0)$}
			\ENDFOR
			\RETURN
		\end{algorithmic}
	\end{algorithm}
	The choice of $C_n = H(x_n,x_{n+1/2})$ has been already investigated in \cite{Alotaibi_2015_best}. There it has been shown that this choice allows to achieve strong convergence of the constructed sequence $\{x_{n}\}_{n\in \mathbb{N}}$ under relatively mild conditions.

	Our aim is to propose and investigate other choices of $\hsubset_n$
	%. In particular, we introduce $C_n$ 
	defined with the help of not only $x_{n},\ x_{n+1/2}$ but also $x_{n-1}$ and/or $x_{n-1+1/2}$. For such choices of $C_n$ with memory the Iterative scheme \ref{alg:generic_best} becomes an iterative scheme with memory, i.e. in the construction of the next iterate $x_{n+1}$    not only current iterate $x_n$ but also $x_{n-1}$ is taken into account. 
	In the sequel we refer to the Iterative Scheme 
	\ref{alg:generic_best} with $C_{n}= H(x_n,x_{n+1/2})$
	as a scheme without memory and we compare it  with
	Iterative Scheme \ref{alg:generic_best}, where $C_{n}$ are with memory (see Proposition \ref{prop:setdefinitions} below).

	The Fej{\'e}rian step in Iterative Scheme \ref{alg:generic_best} coincides with what has been defined in Iterative Scheme \ref{alg:generic_fejer} and was previously discussed in \cite{Alotaibi_2014_sucessive_fejer_approximation,Combettes_2009_fejer_monotonicity}.
	
	Convergence properties of sequences  $\{x_n\}_{n\in\mathbb{N}}$ generated by
	Iterative Scheme \ref{alg:generic_best} are summarized in  Proposition \ref{prop:Haugazeau}   %shows the strong convergence of $\{x_n\}_{n\in\mathbb{N}}$ under relatively unrestrictive conditions, c.f. \cite{weak_to_strong_2001, forcing_strong_2000} and 
	based on Proposition 2.1. of \cite{Alotaibi_2015_best} which, in turn, is based on Proposition 3.1. of \cite{strong_convergence_fo_block}. 
	
	%Please note, that the special case of \ref{alg:generic_best} is when $G_n=C_n=H_n$ considered in \cite{Alotaibi_2015_best}.
	
	\begin{proposition}\label{prop:Haugazeau}
		Let $\solset$ be a nonempty closed convex subset of $H\times G$ and let $x_0=(p_0,v_0^*) \in H\times G$. 
		Let $\{\hsubset_n\}_{n\in \mathbb{N}}$ be any sequence satisfying $Z\subset \hsubset_n\subset H(x_{n},x_{n+1/2})$, $n\in \mathbb{N}$. For the sequence $\{x_n\}_{n \in \mathbb{N}}$ generated by 
		Iterative Scheme \ref{alg:generic_best} the following hold:
		\begin{enumerate}
			\item \label{assertion_H1} $\solset\subset H(x_0,x_n) \cap \hsubset_n$ for $n\in \mathbb{N}$,
			\item \label{assertion_H2} $ \|x_{n+1}-x_0\|\geq \|x_n-x_0\|$ for $n\in \mathbb{N}$,
			\item \label{assertion_H3} $\sum\limits_{n=0}^{+\infty}\|x_{n+1}-x_n\|^2<+\infty$,
			\item \label{assertion_H4} $\sum\limits_{n=0}^{+\infty}\|x_{n+1/2}-x_n\|^2<+\infty$.
			\item \label{assertion_H5} If
			\begin{displaymath}
				\forall x\in H\times G\ \forall \{k_n\}_{n\in \mathbb{N}}\subset \mathbb{N} \quad x_{k_n}\rightharpoonup x \implies x \in \solset,
			\end{displaymath}
			then $x_n\rightarrow P_\solset(x_0)$. 
		\end{enumerate}	
	\end{proposition}
	
	\begin{proof} The proof follows the lines of the proof of Proposition 2.1 of \cite{Alotaibi_2015_best}. The proof of assertion \ref{assertion_H3} and \ref{assertion_H5} coincide with the respective parts of the proof of Proposition 2.1 of \cite{Alotaibi_2015_best} and is omitted here. We provide the proofs of assertions \ref{assertion_H1}, \ref{assertion_H2}, \ref{assertion_H4} for completeness.
		\begin{enumerate}
			\item First we show that $\solset\subset H(x_0,x_n)$. For $n=0$, $x_{1}=P_{H(x_0,x_0)\cap \hsubset_0}(x_0)$, so $\solset\subset H(x_0,x_1)$. Furthermore, $ H(x_0,x_n)\cap  \hsubset_n \subset H(x_0,P_{H(x_0,x_n) \cap \hsubset_n}(x_0))$ and
			\begin{align*}
				\solset\subset H(x_0,x_n)&\implies \solset\subset H(x_0,x_n) \cap \hsubset_n\\
				&\implies \solset\subset H(x_0,P_{H(x_0,x_n)\cap \hsubset_n}(x_0))\\
				& \Leftrightarrow \solset\subset H(x_0,x_{n+1})
			\end{align*}
			\item By construction, for $n\in \mathbb{N}$, $x_{n+1}=P_{H(x_0,x_{n})\cap \hsubset_n}(x_0)$ and $x_{n+1}\subset H(x_0,x_n) \cap \hsubset_n$, so $x_{n+1}\in H(x_0,x_n)$. This implies $\|x_n-x_0\|\leq \|x_{n+1}-x_0\|$.
			%		\item\label{assertionProof3} 	From the fact  that $ \solset\subset H(x_0,x_n)$ we deduce $\|x_n-x_0\|\leq \|P_\solset(x_0)-x_0\|$, $n\in \mathbb{N}$. Hence, the sequence $\{\|x_n-x_0\|\}_{n\in \mathbb{N}}$ converges and
			%		\begin{displaymath}
			%		\lim_{n\rightarrow +\infty} \|x_n-x_0\|\leq \|P_\solset(x_0)-x_0\|.
			%		\end{displaymath}
			%		Since $x_{n+1}\subset H(x_0,x_n)$,
			%		\begin{align*}
			%		\|x_{n+1}-x_n\|^2 & \leq \|x_{n+1}-x_n\|^2 +2\langle x_{n+1}-x_n \ |\ x_n-x_0\rangle \\
			%		&= \|x_{n+1}-x_0\|^2- \|x_n-x_0\|^2.
			%		\end{align*}
			%		Therefore
			%		\begin{displaymath}
			%		\sum\limits_{n=0}^{N} \|x_{n+1}-x_n\|^2 \leq  \sum\limits_{n=0}^{N} (\|x_{n+1}-x_0\|^2- \|x_n-x_0\|^2)= \|x_{n+1}-x_0\|^2
			%		\end{displaymath}
			%		and
			%		\begin{displaymath}
			%		\sum\limits_{n=0}^{+\infty} \|x_{n+1}-x_n\|^2\leq \lim_{n\rightarrow +\infty} \|x_{n+1}-x_0\|^2 \leq \|P_\solset(x_0)-x_0\|<+\infty.
			%		\end{displaymath}
			\item[4.] $\hsubset_n \subset H(x_n,P_{\hsubset_n}(x_n))\subset H(x_n, x_n+\lambda_n (P_{\hsubset_n}(x_n)-x_n) )  = H(x_n,x_{n+1/2})$.
			Since $x_{n+1}\in \hsubset_n \subset \secondhsubset_n \subset H(x_n,x_{n+1/2})$,
			%, where second inclusion follows directly from the projection onto closed convex set in Hilbert space.  Hence
			we deduce that
			\begin{align*}
				&\|x_{n+1/2}-x_n\|^2\leq x_{n+1}-x_{n+1/2}\|^2+\|x_{n+1/2}-x_n\|^2\\
				&\leq  x_{n+1}-x_{n+1/2}\|^2+2\langle x_{n+1}-x_{n+1/2} \ |\ x_{n+1/2}-x_{n}\rangle+\|x_{n+1/2}-x_n\|^2\\
				& \leq \|x_{n+1}-x_n\|^2.
			\end{align*}
			By item \ref{assertion_H3}, $\sum\limits_{n=0}^{+\infty} \|x_{n+1}-x_n\|^2< +\infty$, hence $\sum\limits_{n=0}^{+\infty} \|x_{n+1/2}-x_n\|^2 < +\infty$.
			%		\item Let $x$ be a weak sequential cluster point of $\{x_n\}_{n\in \mathbb{N}}$, $x_{k_n}\rightharpoonup x$. By weak lower semicontinuity of $\|\cdot\|$ we have
			%		\begin{equation*}
			%		\|x_0-x\|\leq \underline{\lim} \|x_0-x_{k_n}\|\leq \|x_0-P_\solset(x_0)\|= \inf_{y\in \solset} \|x_0-y\|. 
			%		\end{equation*}
			%		From assumption $x\in \solset$ and $x=P_\solset(x_0)$ is the only weak sequential cluster point of $\{x_n\}_{n\in \mathbb{N}}$, $x_n\rightharpoonup P_\solset(x_0)$. Thus $x_n-x_0\rightharpoonup P_\solset(x_0)-x_0$ and $\|x_n-x_0\|\rightarrow \|P_\solset(x_0)-x_0\|$. By the theorem of weak-strong convergence we obtain $x_n\rightarrow P_\solset(x_0)$.
		\end{enumerate}
	\end{proof}

	\begin{remark}
		Note that for $\hsubset_n=H(x_n,x_{n+1/2})$ and  $H_n=H_{a_n,b_n^{*}}$ we obtain the primal-dual best approximation algorithm  introduced by Alotaibi et al. in \cite{Alotaibi_2015_best}, involving projections onto the intersections of two halfspaces 
		$H(x_{0},x_{n})\cap H(x_n,x_{n+1/2})$ studied in \cite[Section 28.3]{Bauschke_2011_convex_analysis}.
		Condition $Z\subset \hsubset_n\subset H(x_{n},x_{n+1/2})$, $n\in \mathbb{N}$ allows one to consider
		choices of $C_{n}$ other than  $\hsubset_n=H(x_n,x_{n+1/2})$.
	\end{remark}

	When  $H_n:=H_{a_n,b_n^{*}}$ $n\in \mathbb{N}$, where   $H_{a_n,b_n^{*}}$ are defined by (\ref{settings2}), Proposition \ref{prop:Haugazeau} takes the following form. 
	\begin{proposition}\label{theorem:best_projection}
		Let $\solset$ be a nonempty closed convex subset of $H\times G$ and let $x_0=(p_0,v_0^*) \in H\times G$. 
		Let $\{\hsubset_n\}_{n\in \mathbb{N}}$ be a
		sequence of closed convex sets 
		satisfying the condition
		$Z\subset \hsubset_n\subset H(x_{n},x_{n+1/2})$  %sequences 
		%defined by formulas %(\ref{set:halfspace1})-(\ref{set:halfspace3}) of Proposition \ref{prop:setdefinitions} 
		and
		$H_n:=H_{a_n,b_n^{*}}$, $\in \mathbb{N}$.
		For any sequence $\{x_n\}_{n \in \mathbb{N}}$ generated by 
		Iterative Scheme \ref{alg:generic_best} the following hold:
		\begin{enumerate}
			\item $ \|x_{n+1}-x_0\|\geq \|x_n-x_0\|$ for all $n\in \mathbb{N}$,
			\item $\sum\limits_{n=0}^{+\infty} \|p_{n+1}-p_n\|^2 <+\infty$ and $\sum\limits_{n=0}^{+\infty} \|v_{n+1}^*-v_n^*\|^2 <+\infty$,
			%\item $\sum_{n=0}^{+\infty} \|s_n^*\|^2 < +\infty $ and $\sum_{n=0}^{+\infty} \|t_n\|^2 < +\infty $
			\item $\sum\limits_{n=0}^{+\infty} \|p_n-a_n\|^2 < +\infty$ and $\sum\limits_{n=0}^{+\infty} \|Lp_n-b_n\|^2 < +\infty$,
			\item $p_n\rightarrow \bar{x}$, $v_n^*\rightarrow \bar{v}^*$ and $(\bar{p},\bar{v}^*)\in \solset$.
		\end{enumerate}
	\end{proposition}
	\begin{proof} %To make the paper self-contained we provide	The proof is similar to the proof of \cite[Proposition 3.5]{Alotaibi_2014_sucessive_fejer_approximation}. We provide it here for completeness.
		\mbox{\ }
		\begin{enumerate}
			\item The statement follows directly from item \ref{assertion_H2} of Proposition \ref{prop:Haugazeau}.
			\item By Proposition \ref{prop:Haugazeau},
			\begin{displaymath}
				\sum\limits_{n=0}^{+\infty} \|p_{n+1}-p_n\|^2+\sum\limits_{n=0}^{+\infty} \|v_{n+1}^*-v_n^*\|^2=\sum\limits_{n=0}^{+\infty}\|x_{n+1}-x_n\|^2<+\infty.
			\end{displaymath}
			\item The proof is similar to the proof of \cite[Proposition 3.5]{Alotaibi_2014_sucessive_fejer_approximation}.
			% similar proofs
			\item		  The proof is similar to the proof of \cite[Proposition 3.5]{Alotaibi_2014_sucessive_fejer_approximation}.
		\end{enumerate}		
	\end{proof}
	
	\begin{remark}
		\label{remark_6}
		Proposition \ref{theorem:best_projection} shows the importance of the condition $Z\subset \hsubset_n\subset H(x_{n},x_{n+1/2})$
		in proving the strong convergence of Iterative Scheme \ref{alg:generic_best}.
	\end{remark}
	
	\section{The choice of $C_{n}$}\label{sec:choice}
	One of the main contributions of the paper is to consider $\hsubset_n$ which use the information from the previous step. In this way Iterative Scheme \ref{alg:generic_best} becomes a scheme with memory in the sense that the construction of $x_{n+1}$ depends not only on $x_{n+1/2},\ x_{n},$\ but also on $x_{n-1+1/2},\ x_{n-1}$.
	
	We start with the following propositions.
	
	\begin{proposition}\label{properties:halfspacecombination}
		Let $x,u,v\in H$. Then $H(x,u)\cap H(x,v)\subset H(x,\tau u + (1-\tau)v)$ for all $\tau \in [0,1]$.
	\end{proposition}
	
	\begin{proof} Let $h\in H(x,u)\cap H(x,v)$, i.e
		\begin{displaymath}
			\langle h-u \ |\ x - u\rangle \leq 0 \quad \text{and} \quad \langle h-v \ |\ x - v\rangle \leq 0.
		\end{displaymath}
		For any $\tau \in [0,1]$ we have
		\begin{align*}
			& \langle h - \tau v - (1-\tau)w \ |\ x-\tau v - (1-\tau)w \rangle \\
			& = \langle h \ |\ x \rangle - \tau \langle h \ |\ v \rangle - \tau \langle v \ |\  x \rangle - (1-\tau) \langle h \ |\  w \rangle - (1-\tau) \langle w \ |\  x \rangle\\
			&+ \tau^2 \langle v \ |\ v \rangle +(1-\tau)^2 \langle w \ |\ w \rangle +2\tau (1-\tau)\langle v \ |\ w\rangle \\
			& = \tau \langle h \ |\ x \rangle - \tau \langle h \ |\ v \rangle - \tau \langle v \ |\  x \rangle + \tau \langle v \ |\ v \rangle \\
			& +(1-\tau )  \langle h \ |\ x \rangle - (1-\tau) \langle h \ |\  w \rangle - (1-\tau) \langle w \ |\  x \rangle + (1-\tau) \langle w \ |\ w \rangle \\
			&		+ \tau^2 \langle v \ |\ v\rangle + (1-\tau)^2 \langle w \ |\ w\rangle + 2 \tau (1-\tau) \langle v \ |\ w \rangle - \tau \langle v \ |\ v \rangle - (1-\tau)\langle w \ |\ w \rangle \\
			& \leq \tau^2 \langle v \ |\ v\rangle + (1-\tau)^2 \langle w \ |\ w\rangle + 2 \tau (1-\tau) \langle v \ |\ w \rangle - \tau \langle v \ |\ v \rangle - (1-\tau)\langle w \ |\ w \rangle\\
			&= \tau (\tau - 1) \langle v \ |\ v\rangle + (1-\tau)(-\tau) \langle w \ |\ w \rangle + 2\tau (1-\tau) \langle v \ |\ w \rangle \\
			&\leq \tau (\tau - 1) \|v\|^2 + (1-\tau)(-\tau) \| w\|^2 + 2 \tau(1-\tau) \|v\| \|w\| \\
			& = \tau (\tau -1) ( \|v\|^2 -2 \|v\| \|w\| + \|w\|^2 )= - \tau (1-\tau) (\|v\|+\|w\|)^2\leq 0.
		\end{align*}
		Thus $h \in H(x,\tau u + (1-\tau)v)$.
	\end{proof}
	
	The following proposition provides examples of 
	sets $C_{n}$ with memory satisfying requirements of Proposition \ref{theorem:best_projection} (see Remark \ref{remark_6}).
	
	\begin{proposition}\label{prop:setdefinitions}
		For $\hsubset_n$ defined as 
		\begin{align}
			&\hsubset_n :=   H(x_n, x_{n+\frac{1}{2}}) \cap H(x_{n-1}, x_{n-\frac{1}{2}})\ \text{for}\ n\geq 1 \ \text{and} \ \hsubset_0=H(x_0,x_{1/2}),\label{set:halfspace1}\\
			&\hsubset_n :=   H(x_n, x_{n+\frac{1}{2}}) \cap H(x_{0}, x_{n-1})  \ \text{for} \ n\geq 1 \ \text{and} \  \hsubset_0=H(x_0,x_{1/2}),\label{set:halfspace2}\\
			&\begin{aligned}\label{set:halfspace3}
				\hsubset_n :=&  H(x_n, x_{n+\frac{1}{2}}) \cap H(x_0, \tau_n x_n + (1-\tau_n) x_{n-1}) )  \ \text{for}\ \tau_n \in (0,1),  n\geq 1 \\ &\text{and}\  \hsubset_0=H(x_0,x_{1/2})
			\end{aligned}
			%\item $\hsubset_n =  H(x_0, \tau x_n + (1-\tau) x_{n-1}) \cap H(x_n, x_{n+\frac{1}{2}}) \cap H(x_{n-1}, x_{n-\frac{1}{2}})$
		\end{align}
		%\end{enumerate}
		the assertions 1-5 of Proposition \ref{prop:Haugazeau} holds.
		%	With the $\hsubset_n$ given by formula \ref{set:halfspace1}, \ref{set:halfspace2} and \ref{set:halfspace3}, sets $\hsubset_n$ are closed convex and .
	\end{proposition}
	
	\begin{proof}
		To apply Proposition \ref{prop:Haugazeau} we need only to show that $\hsubset_n$ are closed and convex and $Z\subset \hsubset_n\subset H_n$. The sets $\hsubset_n$ are closed and convex as intersections of finitely many closed halfspaces. By construction of $x_{n+1/2}$ we have $Z\subset H(x_n,x_{n+1/2})$ for all $n \in \mathbb{N}$.
		
		%The proof consists of veryfining that the proposed $C_n$ satisfy the condition $Z\subset \hsubset_n \subset H_n$. It is obvious in case when $\hsubset_n =H_n$. 
		%We will prove the assertion in case of Inertial algorithm.
		\begin{enumerate}
			\item For $\hsubset_n$ given by (\ref{set:halfspace1}) we have $Z \subset H(x_n,x_{n+1/2})\cap H(x_{n-1},x_{n-1+1/2})$ since $Z\subset H(x_n,x_{n+1/2})$.
			\item Let $\hsubset_n$ be given by (\ref{set:halfspace2}).	By construction, $Z\subset H(x_0,x_{1/2})=H(x_0,x_1)=\hsubset_0$. 
			Let $n \in \mathbb{N}$ and suppose $Z\subset \hsubset_k=H(x_{k},x_{k+1/2})\cap H(x_{0},x_{k-1})$ for all $1\leq k\leq n$. We have 
			\begin{align*}
				& Z\subset  H(x_0,x_{n-1})\cap  H(x_{n-1},x_{n-1+1/2})\cap H(x_{0},x_{n-2})= \hsubset_{n}\cap H(x_{0},x_{n-2})\\ % \\
				& \implies Z \subset H(x_0,P_{\hsubset_{n}\cap H(x_{0},x_{n-2})}(x_0))\\
				& \Leftrightarrow Z \subset H(x_0,x_n) 	\Leftrightarrow Z \subset H(x_0,x_n)\cap H(x_{n+1},x_{n+1+1/2})=\hsubset_{n+1}.
			\end{align*}
			By induction, $Z\subset \hsubset_n$ for all $n\geq 0$.
			\item Let $\hsubset_n$ be given by  (\ref{set:halfspace3}).
			By construction, $Z\subset H(x_0,x_{1/2})=H(x_0,x_1)=\hsubset_0$. By Proposition \ref{properties:halfspacecombination}, we have
			\begin{align*}
				& Z\subset H(x_0,x_n) \cap  H(x_0,x_{n-1}) \\
				& \implies  Z \subset H(x_0,\tau_n x_n + (1-\tau_n)x_{n-1})\  \text{for all}\ \tau_n \in (0,1) .
			\end{align*}
			Let $n \in \mathbb{N}$ and suppose $Z\subset \hsubset_k=H(x_{k},x_{k+1/2})\cap H(x_{0},x_{k-1})$ and $Z\subset H(x_0,x_k)$ for all $1\leq k\leq n$. Then
			\begin{align*}
				& Z \subset \hsubset_n =  H(x_n, x_{n+1/2}) \cap H(x_0, \tau_n x_n + (1-\tau_n) x_{n-1}) )\\
				& \implies  Z \subset 	H(x_n, x_{n+1/2}) \cap H(x_0, \tau_n x_n + (1-\tau_n) x_{n-1}) ) \cap H(x_0,x_n)\\
				& \implies Z\subset H(x_0,P_{H(x_0,H(x_n, x_{n+1/2}) \cap H(x_0, \tau_n x_n + (1-\tau_n) x_{n-1}) ) \cap H(x_0,x_n)}(x_0))\\
				& \Leftrightarrow Z \subset H(x_0,P_{\hsubset_n\cap H(x_0,x_n)}(x_0)) = H(x_0,x_{n+1})\\
				& \Leftrightarrow Z\subset H(x_0,x_n)\cap H(x_{n+1})\implies Z \subset H(x_0,\tau_{n+1}x_{n+1}+(1-\tau_{n+1})x_n).
			\end{align*}
			Thus $Z\subset \hsubset_n$ for all $n\in \mathbb{N}$.
		\end{enumerate}
		
	\end{proof}

	%.....
	
	\subsection{Closed-form expressions for projectors onto intersection of three halfspaces}

	In this subsection we recall the closed-form formulas for projectors onto polyhedral sets as given in \cite{explicit_formulas_KR}.
	These halfspaces are given in a form
	\begin{equation} \label{eq:halfspace}
		A_i=\{h \in H\times G \ |\ \langle h \ |\ u_i\rangle \leq \eta_i \},\ i=1,\dots,m
	\end{equation}
	where $u_i\neq 0$, $\eta_i\in\mathbb{R}$, $i=1,\dots,m$, $m\in\mathbb{N}$. 
	
	Let $w_i:=\langle x \ |\ u_i \rangle -\eta_i$, $i \in M:=\{1,\dots,m\}$ and let $G:=[\langle u_i \ |\ u_j \rangle ]_{i,j\in M}$. For any sets $I\subset M$, $J\subset M$, $I,J\neq \emptyset$ the symbol $G_{I,J}$ denote the submatrix of $G$ composed by rows indexed by $I$ and columns indexed by $J$ only.
	Let $s_I(a):=\{ b \in I \ |\ b\leq a  \}$. We define
	\begin{equation*}
		B_I^a:=\left\{\begin{array}{lcl}
			(-1)^{|s_I(a)|} & \text{if} & a \in I,  \\
			(-1)^{|I|+1} & \text{if} & a \notin I.
		\end{array}\right.
	\end{equation*}
	
	\begin{theorem} (\cite[Theorem 2]{explicit_formulas_KR}) \label{prop:projection_conditions}
		Let $m\in \mathbb{N},$ $m\neq0$ and let $M=\{1,\dots,m\}$. Let $A=\bigcap\limits_{i=1}^m A_i\neq\emptyset$, $x\notin A$. Let $\text{rank}\ G=k$. Let $\emptyset \neq I\subset M$, $|I|\leq k$ be such that $\det G_{I,I}\neq 0$. Let
		\begin{equation}\label{coeff:plus}
			\nu_i:= \left\{
			\begin{array}{lcl}
				\sum_{j \in I} w_j B_I^{j} B_I^{i}
				\det G_{I\backslash j , I\backslash i} & \text{if} & |I|>1,\\
				w_i & \text{if} & |I|=1
			\end{array}\right.
			\quad \text{for all} \quad i \in I
		\end{equation}
		and, whenever  $ I^\prime:=M\backslash I$ is nonempty, let
		\begin{equation}\label{coeff:notplus}
			\nu_{i^\prime}:= 
			\sum_{j \in I\cup \{i^\prime\}} w_j B_I^{j} B_I^{i^\prime}
			\det G_{I , (I \cup i^\prime)\backslash j  }
			\quad \text{for all} \quad i^\prime \in I^\prime.
		\end{equation}

		If $\nu_i>0$ for $i \in I$ and $\nu_{i^\prime}\leq 0$ for all $i^\prime\in I^\prime$, then 
		\begin{equation}\label{formula:projection}
			P_A(x)=x-\sum\limits_{i \in I} \frac{\nu_i}{\det G_{I,I}}u_{i}.
		\end{equation}
		Moreover, among all the elements of the set $\Delta$ of all subsets $I\subset M$ there exists at least one $I\in \Delta$ for which: (1) $\det G_{I,I}\neq 0$, (2) the coefficients $\nu_i$, $i \in I$ given by \eqref{coeff:plus} are positive, (3) the coefficients $\nu_{i^\prime}$, $i^\prime \in I^{\prime}$ given by \eqref{coeff:notplus} are nonpositive.
	\end{theorem}	
	To obtain the closed-form expression formula for projection of a point on intersection of three halfspaces we propose the following finite algorithm for finding $\nu_i$ as given in formula \eqref{formula:projection}.
	\begin{algorithm}[H]
		\caption{Algorithm for finding ${\nu}=[{\nu}_i]_{i \in\{ 1,2,3\}}$}\label{algorithm:searchnu}
		\begin{algorithmic}
			\STATE{Let $\mathcal{K}$ be a set of all nonempty subsets of K=\{1,2,3\}}
			\WHILE{$\mathcal{K}\neq \emptyset$}
			\STATE{Choose randomly $I\in \mathcal{K}$}
			\IF{$\det G_{I,I}\neq 0$}
			\STATE{Find $\nu=[\nu_i]_{i \in I}$ such that $G_{I,I}\nu=[\langle x \ |\ u_i \rangle - \eta_i]_{i \in I}$}
			\IF{$\nu>0$ }
			\IF{for all $i\in K\backslash I$, $\langle x - \sum_{k\in I} \nu_k u_k \ |\ u_i \rangle - \eta_i  \leq 0$ }
			\STATE{Terminate, put ${\nu}_i=0$ for $i\in K\backslash I$}
			\ENDIF
			\ENDIF
			\ENDIF
			\STATE{${\mathcal{K}}:={\mathcal{K}}\backslash I$}
			\ENDWHILE
		\end{algorithmic}
	\end{algorithm}
	Note that  
	%for all $I\in \cal{K}$ the conditions can be checked independently, and so 
	Iterative Scheme \ref{algorithm:searchnu} can be easily parallelized. For three halfspaces (i.e. $m=3$ in \eqref{eq:halfspace}) at most $7$ subsets $I\in \mathcal{K}$ need to be checked to calculate  coefficients $\nu_i$, $i=1,2,3$ of formula \eqref{formula:projection}. Note that the above defined Iterative Scheme \ref{algorithm:searchnu} can be useful for several algorithms, i.e. for computation of next iterate in \cite{VanHieu2017}.
	
	On the other hand, when considering Iterative scheme 
	\ref{alg:generic_best} with halfspaces  generated as
	$$
	H(x_{0},x_{n})\cap C_{n},
	$$
	where $C_{n}$ are as in 
	Proposition \ref{prop:setdefinitions} the number of iterations can be reduced to 4. This is the content of the following Proposition.
	
	Let  $W:=H(x_0,x_n)\cap H(x_n,x_{n+1/2}) \cap A_{3}$, where $A_3$ is given by one the following
	\begin{equation}\label{setpropB}
		\left\{\begin{array}{l}
			H(x_n,x_{n+1/2}),\\
			H(x_0,x_{n-1}),\\
			H(x_0, \tau_n x_n + (1-\tau_n) x_{n-1}) ,\  \tau_n\in(0,1).
		\end{array}\right.
	\end{equation}
	For simplicity, in Proposition \ref{projection_reduced} we use $H(a,b)=A_{3}$.
	%Hence $W=H(x_0,x_n)\cap H(x_n,n_{n+1/2}) \cap H(a,b)$, where $H(a,b)=A_{3}$.
	% is any of halfspaces from \eqref{setpropB}.
	
	\begin{proposition}
		\label{projection_reduced}
		For finding projection of $x_0$ onto $W$ with the help of Iterative Scheme \ref{algorithm:searchnu} at most 4 subsets $I\in \mathcal{K}$ need to be checked.	
	\end{proposition}
	\begin{proof}
		
		We show that the projection of $x_0$ onto $W$ does not require the  cases $I=\{1\}$, $I=\{3\}$, $I=\{1,3\}$ to be checked.
		\begin{enumerate}
			\item Suppose $I=\{1\}$. Then $	\nu_1=\frac{\langle x_0-x_n \ |\ x_0 - x_n \rangle}{\|x_0-x_n\|^2}=1 $, $\eta_2=\langle x_{n+1/2} \ |\ x_n-x_{n+1/2} \rangle$	and for $2 \in K\backslash I$ we have 
			\begin{align*}
				\langle x_0-\nu_1 (x_0-x_n) \ |\ x_n-x_{n+1/2} \rangle -\eta_2>0.	 
				%x_{n+1}=x_0-\frac{\langle x_0-x_n \ |\ x_0 - x_n \rangle}{\|x_0-x_n\|^2} (x_0-x_n)= x_n
			\end{align*}
			\item Suppose $I=\{3\}$.Then
			\begin{displaymath}
				\nu_3=\frac{\langle x_0-b \ |\ a-b \rangle}{\|a-b\|^2}
			\end{displaymath}
			and for $1 \in K\backslash I$ we have 
			\begin{align}
				&\langle x_0-\nu_3 (a-b) \ |\ x_0-x_n \rangle - \eta_1\notag\\
				&=\langle P_{H(a,b)}(x_0)-x_n \ |\ x_0 - x_n \rangle\notag\\
				&=\langle P_{H(a,b)}(x_0)-x_n \ |\ x_0 -P_{H(a,b)}(x_0)  \rangle\notag\\
				&+\langle P_{H(a,b)}(x_0)-x_n \ |\ P_{H(a,b)}(x_0) - x_n \rangle\geq 0.\label{inequality:1}
			\end{align}
			If equality in \eqref{inequality:1} holds then $P_{H(a,b)}(x_0)=x_n$. Then for $2 \in K\backslash I$ we have 
			\begin{align*}
				&\langle x_0-\nu_3(a-b) \ |\ x_n - x_{n+1/2} \rangle - \eta_2\\
				&=\langle P_{H(a,b)}(a,b)-x_{n+1/2} \ |\ x_n - x_{n+1/2} \rangle \\
				& = \langle x_n-x_{n+1/2} \ |\ x_n - x_{n+1/2} \rangle>0.
			\end{align*}
			\item Suppose, $I=\{1,3\}$. Then
			\begin{align*}
				%\nu_1&=\|x_0-x_n\|^2 \|a-b\|^2 - \langle x_0-b \ |\ a-b \rangle \langle x_0-x_n \ |\ a-b \rangle \\
				\nu_3&=-\|x_0-x_n\|^2 \langle a -b \ |\ x_0-x_n \rangle + \langle x_0-b \ |\ a-b \rangle \|x_0-x_n\|^2\\
				&=\|x_0-x_n\|^2 \langle  x_n - b \ |\ a-b \rangle\leq 0
			\end{align*}
			because $x_n\in H(a,b)$.
			
		\end{enumerate}
		This shows that the choices $I=\{1\}$, $I=\{3\}$, $I=\{1,3\}$ do not lead to suitable projection weights $\nu_i\ge 0$, $i=1,2,3$.
	\end{proof}

	\section{Convergence analysis}\label{sec:convergence_analysis}
	%In this section we investigate convergence properties of sequence $\{x_n\}_{n\in \mathbb{N}}$ constructed in Iterative Scheme \ref{alg:generic_best}.
	In this section we analyse convergence properties of Iterative Scheme \ref{alg:generic_best}. To this aim we	introduce \textit{attraction property}  (Proposition \ref{prop:inertial_property}). The proposed results provide:
	\begin{itemize}
		\item new measure of quality of the solution generated by Iterative Scheme \ref{alg:generic_best}. Note that it was shown in \cite{Alotaibi_2015_best,strong_convergence_fo_block}
		that with every iteration, $x_n$ is further from $x_0$. However, there was no results relating $x_n$ and the solution $P_Z(x_0)$. %The relation is not straightforward because the increasing distance from $x_n$ to $x_0$ does not imply that the  distance from $x_n$ to $P_Z(x_0)$ decrease 
		By \textit{attraction property}, the distance
		from $x_n$ to the solution $P_Z(x_0)$ need not be decreasing, however, $x_{n}$ remain in a ball centred at $P_Z(x_0)$
		with  radius  which is a nonincreasing function of $n$
		(by \eqref{assertionCA_2} of the Proposition \ref{prop:convergence_analysis});
		%By \textit{attraction property}, the distance
		% from $x_n$ to the solution $P_Z(x_0)$ need not be decreasing, however, there exists  a nonincreasing bound of this distance;
		%can be further from $x_0$ but not necessarily closer to the solution $P_Z(x_0)$;
		\item new evaluation criteria allowing to compare algorithms (we use
		them to compare experimentally algorithms with different choices of $C_n$) (Proposition \ref{prop:inertial_property}).
		%\item new stopping criteria allowing to terminate algorithm when  bounds on distance between $x_n$ and the solution are sufficiently small.
	\end{itemize}
	%evaluation criteria for determining the quality of the solution produced by Iterative Scheme \ref{alg:generic_best}.
	
	We start with the following technical lemma.

	\begin{lemma}\label{lemma:ball}
		Let $U$ be a real Hilbert space and let $u_{1},u_{2},u_{3}\in U$,  $u_3\in H(u_1,u_2)$, $w=\frac{1}{2}(u_1+u_3)$, $r:=\|w-u_1\|$. Then
		\begin{enumerate}[(i)]
			\item\label{assertionBall1} $\|w-u_2\|\leq \frac{1}{2}\|u_1-u_3\|$,
			\item\label{assertionBall2} $\|u_2-u_3\|^2\leq b(u_2)$, where $b(\cdot):=4r^2-\|\cdot-u_1\|^2$.
			\item\label{assertionBall3} Moreover, if $u_4\in H$ and $u_2\in H(u_1,u_4)$, then $b(u_2)\leq b(u_4)$.% If $u_4\neq u_2$, then the latter inequality is strict.
		\end{enumerate}

	\end{lemma}
	\begin{proof}\hfill\\[-0.1cm]
		\begin{enumerate}[(i)]
			\item We have
			\begin{align*}
				r&= \| w-u_1\|=\| \frac{1}{2}u_1+\frac{1}{2}u_3-u_1\|=\frac{1}{2}\|u_3-u_1\|\\
				& = \|\frac{1}{2}u_1+\frac{1}{2}u_3-u_3\|=\|w-u_3\|.
			\end{align*}	
			By contradiction, suppose $\|w-u_2\|> \frac{1}{2}\|u_1-u_3\|$. Since $u_3\in H(u_1,u_2)$ we have
			\begin{align*}
				\|u_3-u_2\|^2&+\|u_1-u_2\|^2\\
				&=\|u_3-u_2\|^2+\|u_1-u_2\|^2+2 \langle u_3-u_2 \ |\ u_2 - u_1 \rangle + 2 \langle u_3- u_2 \ |\ u_1 - u_2\rangle\\
				&= \|u_3- u_1 \|^2 +2 \langle u_3- u_2 \ |\ u_1 - u_2\rangle \leq \|u_3-u_1\|^2=4r^2.
			\end{align*}
			On the other hand
			\begin{align*}
				4r^2&=\|u_3-u_1\|^2\geq \|u_3-u_2\|^2+\|u_1-u_2\|^2\\
				&=\|u_3-w \|^2 - 2 \langle u_3-w \ |\ u_2 - w \rangle + \|u_2 - w \|^2\\
				&+\|u_1-w\|^2 -2 \langle u_1-w \ |\ u_2 - w \rangle + \|u_2-w\|^2\\
				& = 2r^2+2 \|u_2-w \|^2 -2 \langle u_3+u_1-2w \ |\ u_2 - w \rangle > 4r^2,
			\end{align*}
			a contradiction.
			\item We have
			\begin{align*}
				\|u_3-u_2\|^2 &= \langle u_3-u_2 \ |\ u_3-u_2 \rangle = \langle u_3-u_2 \ |\ u_1 - u_1 + u_3-u_2 \rangle\\
				& = \langle u_3-u_2 \ |\ u_1-u_2 \rangle + \langle u_3-u_2 \ |\ u_3- u_1 \rangle \\
				& \leq \langle u_3-u_2 \ |\ u_3- u_1 \rangle = \langle u_3-u_2-u_1+ u_1 \ |\ u_3-u_1 \rangle\\
				&= \langle u_3-u_1 \ |\ u_3-u_1 \rangle + \langle u_1- u_2 \ |\ u_3- u_1 \rangle\\
				& = 4r^2 + \langle u_1- u_2 \ |\ u_2 - u_2 + u_3- u_1 \rangle\\
				& = 4r^2 +  \langle u_1 - u_2 \ |\ u_2-u_1 \rangle + \langle u_1-u_2 \ |\ u_3-u_2 \rangle\\ &\leq 4r^2 - \|u_1-u_2\|^2.
			\end{align*}
			\item The assertion \eqref{assertionBall3} stems from the fact that $u_2\in H(u_1,u_4)$ implies $\|u_1-u_2\|^2\geq \|u_1-u_4\|^2$ and 
			\begin{equation*}
				\|u_3-u_2\|^2\leq \| u_1-u_3\|^2 - \|u_1-x_{n}\|^2\leq \| u_1-u_3\|^2 - \|u_1-u_4\|^2.
			\end{equation*}
			%If for some $n\geq 1$ we have $u_4\neq u_2$, then $\|u_1-u_2\|^2 > \|u_1-u_4\|^2$.
		\end{enumerate}\mbox{\ }
	\end{proof}

	%Let $x_0\in H\times G$ be a starting point of Iterative Scheme \ref{alg:generic_best}. Since $Z$ is closed and convex there exists a unique  $\bar{x}\in Z$
	%such that $\text{dist\,}(x_{0},Z)=\|x_{0}-\bar{x}\|$. 

	We show that all the points $x_{n}$, $n\in\mathbb{N}$, generated by Iterative Scheme \ref{alg:generic_best} are contained in the ball centred at $w:=\frac{1}{2}(x_0+\bar{x})$
	with radius $r:=\|w-x_0\|=\frac{1}{2}\text{dist\,}(x_{0},Z)$ and the distance from $x_n$ to the solution $\bar{x}$ is bounded from above by a nonincreasing sequence.

	\begin{proposition}\label{prop:convergence_analysis}
		Let $x_0\in H\times G$. Any sequence $\{x_n\}_{n\in \mathbb{N}}$ generated by Iterative Scheme \ref{alg:generic_best} satisfies the following:
		\begin{enumerate}[(i)]
			\item \label{assertionCA_1}
			$\|w-x_n\|\leq \frac{1}{2}\|x_0-\bar{x}\|$, $n\in \mathbb{N}$.
			\item \label{assertionCA_2} $\|x_n-\bar{x}\|^2\leq b_n
			$, where $b_n:=4r^2-\|x_n-x_0\|^2\geq 0$, $n\in \mathbb{N}$.
			\item \label{assertionCA_3} Moreover, if $x_n\in H(x_{0},x_{n-1})$ for all $n\geq 1$, the sequence $\{b_n\}_{n\in \mathbb{N}}$ is nonincreasing. If for some $n\ge 1$ we have $x_{n-1}\neq x_{n}$, then
			\begin{align}
				\begin{aligned}
					&\|\bar{x}-x_n\|^2< \| x_0-\bar{x}\|^2 - \|x_0-x_{n-1}\|^2,  \label{eq:}\\
					& b_n <  b_{n-1}.
				\end{aligned}
			\end{align}
		\end{enumerate}
		
	\end{proposition}
	\begin{proof}
		Let $n\in \mathbb{N}$. We have $\bar{x}\in H(x_0,x_n)$.  We obtain \eqref{assertionCA_1} and \eqref{assertionCA_2} by applying Lemma \ref{lemma:ball}  with $u_1=x_0$, $u_2=x_n$ and $u_3=\bar{x}$.
		\par The assertion \eqref{assertionCA_3} follows from \eqref{assertionBall3} of Lemma \ref{lemma:ball} with $u_1=x_0$, $u_2=x_n$, $u_3=\bar{x}$ and $u_4=x_{n-1}$.
		Moreover, if for some $n\geq 1$ we have $x_{n-1}\neq x_{n}$, then $\|x_0-x_n\|^2 > \|x_0-x_{n-1}\|^2$, which follows from \ref{assertion_H2} of Proposition \ref{prop:Haugazeau}.
	\end{proof}
	
	Let us note that Iterative Scheme \ref{alg:generic_best} is sufficiently general to encompass algorithm 2.1 of \cite{Alotaibi_2015_best} as well as any algorithm 
	with  memory introduced by $C_n$ satisfying the requirements of Proposition \ref{theorem:best_projection}.
	%$C_n$ (see Proposition \ref{prop:setdefinitions}) as  special cases. 
	In consequence, Proposition \ref{prop:convergence_analysis} provides properties of sequences $\{x_n\}_{n\in \mathbb{N}}$ constructed in these algorithms. 
	
	\begin{corollary}
		For any $x_0\in H\times G$ and 
		$x_n,\ x_{n+1/2}$, $n\geq 1$ generated by Iterative Scheme \ref{alg:generic_best}  we have
		\begin{equation*}
			\|x_{n+1/2}-\bar{x}\|^2\leq \|x_0-\bar{x}\|^2-\|x_n-x_0\|^2-\|x_{n+1/2}-x_n\|^2
		\end{equation*}
		for any $\bar{x}\in H(x_0,x_n)\cap H(x_n,x_{n+1/2})$.
	\end{corollary}
	\begin{proof}
		Let $n\in \mathbb{N}$. Applying (ii) of Lemma \ref{lemma:ball} to $u_1=x_0$, $u_2=x_n$, $u_3=\bar{x}$ 
		\begin{equation}\label{ineq:1}
			\|x_n-\bar{x}\|^2\leq 
			\|x_0-\bar{x}\|^2-\|x_n-x_0\|^2.
		\end{equation}
		Applying again \eqref{assertionBall2} of Lemma \ref{lemma:ball} to $u_1=x_n,\ u_2=x_{n+1/2}, u_3=\bar{x}$ we obtain
		\begin{equation}\label{ineq:2}
			\|x_{n+1/2}-\bar{x}\|^2\leq \|x_n-\bar{x}\|^2-\|x_{n+1/2}-x_n\|^2.
		\end{equation}
		In consequence, we have $\|x_{n+1/2}-\bar{x}\|\leq \|x_n-\bar{x}\|$.	Combining \eqref{ineq:1} and \eqref{ineq:2} we obtain
		\begin{equation*}
			\|x_{n+1/2}-\bar{x}\|^2\leq \|x_n-\bar{x}\|^2-\|x_{n+1/2}-x_n\|^2\leq \|x_0-\bar{x}\|^2-\|x_n-x_0\|^2 -\|x_{n+1/2}-x_n\|^2.
		\end{equation*}
		$\;$			
	\end{proof}
	
	To prove Proposition \ref{prop:inertial_property}, which is our main result in this section we need the following Lemma.
	\begin{lemma}\label{lemma:inertial_property}
		Let $U$ be a real Hilbert space and let $D\subset U$ be a nonempty subset of $U$. Let $u_{1},u_{2},u_{3},u_{4}\in U$ and  $u_3\in H(u_1,u_2) \cap H(u_2,u_4) \cap D$, $w=\frac{1}{2}(u_1+u_3)$, $r:=\|w-u_1\|$. 
		
		Let $\bar{q}\in H(u_1,u_2)\cap H(u_2,u_4)\cap D$. Then $\bar{q} \in H(u_1,q)$, where  $q=Q(u_1,u_2,u_4):=P_{H(u_1,u_2)\cap H(u_2,u_4)}(u_1)$ and
		\begin{align*}
			& \|u_1-\bar{q}\|^2\geq \|u_1-q\|^2+\|\bar{q}-q\|^2,\\
			& \|u_3-\bar{q} \|^2 \leq 4r^2 - \|u_1-\bar{q}\|^2 \leq 4r^2- \|u_1-q\|^2-\|\bar{q}-q\|^2.
		\end{align*}
		
		\begin{proof}
			It is immediate that $\bar{q}\in H(u_1,q)$. Thus
			\begin{displaymath}
				\|u_1-\bar{q}\|^2=\|u_1-q\|^2 + 2\langle u_1-q \ |\ q - \bar{q} \rangle + \|\bar{q}-q\|^2\geq \|u_1-q\|^2+\|\bar{q}-q\|^2 .
			\end{displaymath}
			By Lemma \ref{lemma:ball}, since $u_3\in H(u_1,\bar{q})$,  we have $\|u_3-\bar{q}\|\leq 4r^2-\|u_1-\bar{q}\|^2$ which completes the proof.
		\end{proof}
		
	\end{lemma}
	
	To compare best approximation algorithms as defined in \cite{Alotaibi_2015_best} with the  Iterative Scheme \ref{alg:generic_best} with memory we concentrate on single step gains. To this end let us denote $q_n:=P_{D(n)}(x_0)$, $x_n:=P_{D(n-1,n)}(x_0)$,
	where $D(n)=H(x_0,x_{n-1})\cap H(x_{n-1},x_{n-1+1/2})$ as e.g in \cite{Alotaibi_2015_best}  and $D(n-1,n)=H(x_0,x_{n-1})\cap H(x_{n-1},x_{n-1+1/2})\cap C_{n-1}$% is given as in Iterative Scheme \ref{alg:generic_best} 
	with $C_{n-1}$ as in Proposition \ref{prop:setdefinitions}.

	\begin{proposition} \label{prop:inertial_property} [\textit{Attraction property}]
		Sequences $\{x_n\}_{n\in \mathbb{N}}$,
		$\{x_{n+1/2}\}_{n\in \mathbb{N}}$ generated by Iterative Scheme  \ref{alg:generic_best} satisfy the following:
		\begin{enumerate}[(i)]
			\item $\|x_0-x_{n}\|^2\geq \|x_0-q_n\|^2+\|x_{n}-q_n\|^2$,
			\item\label{prop:inertial_property:item2} $\|\bar{x}-x_{n}\|^2 \leq \|x_0-\bar{x}\|^2-\|x_0 - x_n \|^2  \leq \|x_0-\bar{x}\|^2-\|x_0 - q_n \|^2 - \|x_{n}-q_n\|^2$,
		\end{enumerate}
		where $q_n:=P_{H(x_0,x_{n-1})\cap H(x_{n-1},x_{n-1+1/2})}(x_0)$ and $\bar{x}=P_Z(x_0)$.
		
	\end{proposition}
	
	\begin{proof}
		
		The proof follows directly from Lemma \ref{lemma:inertial_property} with $u_1=x_0$, $u_2=x_{n-1}$, $u_3=\bar{x}$, $u_4=x_{n-1+1/2}$, $\bar{q}=P_{H(x_0,x_{n-1})\cap C_{n-1}}(x_0)$ and $D=C_{n-1}$.
	\end{proof}
	Let us note that, in the case when $x_n=q_n$, by \eqref{prop:inertial_property:item2}, we have $\|\bar{x}-q_{n}\|^2 \leq \|x_0-\bar{x}\|^2-\|x_0 - q_n \|^2 $. Hence, in the Iterative Scheme \ref{alg:generic_best} we are interested in choices of $C_n$  which make the difference $x_n-q_n$ large. Note that in case of $C_{n-1}=H(x_{n-1}, x_{n-1+1/2})$,  $\|x_{n}-q_n\|^2=0$. For other choices of $C_{n-1}$ the worst case leads to $\|x_{n}-q_n\|^2=0$, however, we can expect some improvement. Consequently,  the proposed \textit{attraction property} may serve as an evaluation criterion for comparing various versions of Iterative Scheme \ref{alg:generic_best}.
	% with the original (non-inertial) one. 
	
	%to the definition of the following stopping criterion (as a measure of distance of current iterate to the solution):
	%\begin{equation}
	%(\|x_0 - q_n \|^2 + \|x_{n}-q_n\|^2) - (\|x_0 - q_{n-1} \|^2 + %\|x_{n-1}-q_{n-1}\|^2)\leq \epsilon \nonumber
	%\end{equation}
	%where $\epsilon >0 $ denotes given tolerance.
	
	\section{Proximal algorithms}\label{sec:the_algorithm}
	
	\par Let $H$ and $G$ be real Hilbert spaces, let $f:\ H\rightarrow (-\infty,+\infty]$ and $g:\ G\rightarrow (-\infty,+\infty]$ be proper lower semicontinuous convex functions and let $L:\ H\rightarrow G$ be a bounded linear operator. Iterative Scheme \ref{alg:proximal_best} defined bellow is an application of Iterative Scheme \ref{alg:generic_best} to optimization problem  \eqref{prob_main}-\eqref{optimization_funD}, i.e. we consider the pair of problems,
	\begin{equation}\label{prob:1}
		\min_{p\in H} F_P(p):=f(p)+g(Lp)
	\end{equation}
	%In Fenchel-Rockafellar model (see \cite{pennanen_dualization})
	and the dual problem to (\ref{prob:1}), % is
	\begin{equation}\label{prob:2}
		\min_{v^*\in G}\ F_D(v^*): = f^*(-L^*v^*)+g^*(v^*).
	\end{equation}
	%where $f^*$ denotes conjugate function \cite{Rockafellar_1970_convex_analysis}.
	If \eqref{prob:1} has a solution $\bar{p}\in H$ and the regularity condition holds, e.g.
	\begin{displaymath}
		0 \in \text{sqri} (\text{dom} \; g-L(\text{dom} \;f)),
	\end{displaymath}
	where $\text{dom}$ denotes the effective domain of a function and for any convex closed set $S$
	\begin{displaymath}
		\text{sqri}S:= \{x\in S\ |\ \bigcup_{\lambda>0} \lambda (S-x)\ \text{is a closed linear subspace of}\ H \},
	\end{displaymath}
	there exists $\bar{v}^*\in G$ solving \eqref{prob:2} and
	\begin{equation}\label{set:Zalgorithm}
		(\bar{p},\bar{v}^*)\in Z=\{ (p,v^*)\in H\times G\ |\ -Lv^*\in \partial f(x)\ \text{and}\ v^*\in \partial g(Lx)  \}.
	\end{equation}
	Conversely, if $(\bar{p},\bar{v}^*)\in Z$, then $\bar{p}$ solves \eqref{prob:1} and $\bar{v}^*$ solves \eqref{prob:2}.
	%Let $\bar{p}\in H$ be a point such that
	%\begin{equation}\label{assertion:primal}
	%0 \in \partial f (\bar{p})+L^* \partial g (L\bar{p})
	%\end{equation}
	%and $\bar{v}^*\in G$ be a point such that
	%\begin{equation}\label{assertion:dual}
	%0\in - L (\partial f)^{-1}(-L\bar{v}^*)+ (\partial f)^{-1}(\bar{v}^*).
	%\end{equation}
	%If such a pair $(\bar{p},\bar{v}^*)\in H\times G$ exists, then $\bar{p}$ is a solution to \eqref{prob:1} and $\bar{v}^*$ is a solution to \eqref{prob:2}. 
	The set $Z$ defined by \eqref{set:Zalgorithm} is of the form \eqref{set:Z}, when $A=\partial f$ and $B=\partial g$. 
	
	%\begin{definition}
	%Let $f:\ H \rightarrow \mathbb{R}$ be a proper, convex lower semi-continuous function and let $x\in H$. 
	Recall that for any $x\in H$ and any proper convex and lower semi-continuous function $f:H \rightarrow \mathbb{R}\cup\{+\infty\}$ the \textit{proximity operator} $\prox_f(x)$ is defined as the unique solution to the optimization problem
	\begin{displaymath}
		\min_{y\in H} \left(f(y) + \frac{1}{2} \|x-y\|^2  \right).
	\end{displaymath}
	%\end{definition}
	\begin{theorem}\cite[Example 23.3]{Bauschke_2011_convex_analysis}
		Let $f: H \rightarrow \mathbb{R}\cup \{+\infty\}$ be a proper convex lower semi-continuous function, $x\in H$ and $\gamma>0$. Then
		\begin{displaymath}
			J_{\gamma \partial f}(x)=\prox_{\gamma f} (x).
		\end{displaymath}
		%	where $J_A(x)=(Id+A)^{-1}(x)$.
	\end{theorem}

	\begin{algorithm}[ht]
		\caption{Proximal primal-dual best approximation iterative scheme \label{alg:proximal_best}}
		\begin{algorithmic}
			\STATE{Choose an initial point $x_0=(p_0,v_0^*)\in H\times G$ and $\varepsilon>0$} 
			\STATE{Choose sequences of parameters $\{\lambda_n\}_{n\ge 0}\in (0,1]$ and $\{\gamma_n\}_{n\ge 0},\{\mu_n\}_{n\ge 0}\in [\varepsilon,1/\varepsilon]$}
			\FOR{$n=0,1\dots$}
			%\STATE{\bf{Fej{\'e}rian step}}
			%\STATE{Choose $\secondhsubset_n$ such that $Z\subset \secondhsubset_n$}
			\STATE{$a_n = \text{Prox}_{\gamma_n f} (p_n - \gamma_n L^* v_n^*)$}
			\STATE{$b_n = \text{Prox}_{\mu_n g} (L p_n + \mu_n v_n^*)$}
			\STATE{$a_n^* = \gamma_n^{-1}(p_n-a_n) - L^* v_n^* $}
			\STATE{$b_n^* = \mu_n^{-1} (Lp_n-b_n) +  v_n^*  $}
			
			\STATE{$s_n^* = (a_n^* + L^*b_n^*, b_n - La_n)$}
			\STATE{$\eta_n =\left\langle a_n \mid a_n^* \right\rangle + \left\langle b_n \mid b_n^* \right\rangle$}
			\STATE{$H_n = \left\{x \in H\times G \mid \left\langle  x \mid s_{a_n,b_n}^*\right\rangle \leq \eta_{a_n,b_n}  \right\}$}
			\IF{$\|s_n^*\|=0$}
			\STATE{$\bar{x}=x_n$, $\bar{v}^*=v_n^*$}
			\STATE{Terminate}
			\ELSE
			\STATE{\textbf{Fej{\'e}rian step}}
			\STATE{$x_{n+1/2}=x_n+\lambda_n(P_{\secondhsubset_n}(x_{n})-x_n)$}
			\STATE{\textbf{Haugazeau step}}
			\STATE{Choose $\hsubset_n$ closed convex such that $Z\subset C_n \subset H(x_n,x_{n+1/2})$%according to Proposition \ref{prop:setdefinitions} 
			} 				
			\STATE{$x_{n+1}=P_{H(x_0, x_n) \cap \hsubset_n}(x_0)$}
			\ENDIF
			\ENDFOR
			\RETURN
		\end{algorithmic}
	\end{algorithm}
	Convergence properties of Iterative Scheme \ref{alg:proximal_best} are summarized in Proposition \ref{theorem:best_projection}.
	
	\subsection{Generalization to finite number of functions}
	Let $M$ and $K$ be natural numbers. Let $E=\bigoplus_{i=1}^M H_i \times \bigoplus_{k=1}^K G_k$, where $H_i$, $G_k$ are real Hilbert spaces, $i=1,\dots,M$, $k=1,\dots,K$. Let $f_i:\ H_i \rightarrow \mathbb{R}\cup \{+\infty \} $ %\in \Gamma(H_i)$
	and $g_k:\ G_i \rightarrow \mathbb{R}\cup \{+\infty \} $ %\in \Gamma(G_k)$
	be proper lower  semicontinuous convex functions and $L_{ik}:\ H_i \rightarrow G_k$ be  bounded linear operators, $i=1,\dots,M$, $k=1,\dots,K$. Consider the primal problem
	\begin{equation}\label{prob:1multi}
		\min_{p_1\in H_1,\dots,p_M \in H_M} \quad \sum_{i=1}^{M} f_i(p_i) + \sum_{k=1}^{K} g_k\left( \sum_{i=1}^{M} L_{ik}p_i\right).
	\end{equation}

	Problem formulation \eqref{prob:1multi} is general enough to cover problem arising in diverse applications including signal and image reconstruction, compressed sensing and machine learning \cite
	{Hendrich_2014_phd}.
	The dual problem to  \eqref{prob:1multi} is
	\begin{equation}\label{prob:2multi}
		\min_{v_1^*\in G_1,\dots,v_K^* \in G_K} \quad \sum_{i=1}^{M} f_i^*\left(-\sum_{k=1}^K L_{ki}^* v_k^* \right) + \sum_{k=1}^{K} g_k^*(v_k^*).
	\end{equation}
	Assume that
	\begin{displaymath}
		(\forall i \in \{1,\dots,M\}) \quad 0 \in \text{ran} \left(\partial f_i + \sum_{k=1}^{K} L_{ki}^* \circ \partial g_k \circ L_{ik}\right),
	\end{displaymath}
	where $\text{ran} D$ denotes the range of an operator $D$.
	
	Then the set
	\begin{align}
		\begin{aligned}	
			\label{set:KTmulti}
			Z:=&\bigg\{ (p_1,\dots,p_M,v_1^*,\dots,v_K^*)\in E \ |\ %\\  &
			%(\forall i \in \{1,\dots, M\}) \quad 
			-\sum_{k=1}^{K} L_{ki}^* v_k^* \in \partial f_i(p_i), %\\  %\quad \text{and}\quad  %&
			%(\forall k \in \{1,\dots,K\}) \quad
			\sum_{i=1}^{M} L_{ik} p_i \in \partial g_k^* (v_k^*), \\
			& i=1\dots,M,\ k=1,\dots,K  \bigg\}
		\end{aligned}
	\end{align}
	is nonempyty and if $(\bar{p}_1,\dots,\bar{p}_M,\bar{v}_1^*,\dots,\bar{v}_K^*)\in Z$ then $(\bar{p}_1,\dots,\bar{p}_M)$ solves (\ref{prob:1multi}) and $(\bar{v}_1^*,\dots,\bar{v}_K^*)$ solves (\ref{prob:2multi}). To find an element of set $Z$ defined by \eqref{set:KTmulti} we propose the Iterative Scheme \ref{alg:proximal_best2}.
	
	\begin{algorithm}[H]
		\caption{Proximal primal-dual best approximation iterative scheme for finite number of functions\label{alg:proximal_best2}}
		\begin{algorithmic}
			\STATE{Choose an initial point $x_0=(p_0,v_0^*) %= (\prod_{i=1}^M p_{i,0},\prod_{k=1}^K v_{k,0}^*)
				\in \bigoplus_{i=1}^M H_i \times \bigoplus_{k=1}^K G_k$  and $\varepsilon>0$,\\ $p_0=(p_{1,0},\dots,p_{M,0})$, $v_0^*=(v_{1,0}^*,\dots,v_{K,0}^*)$} 
			\STATE{Choose sequences of parameters $\{\lambda_n\}_{n\ge 0}\in (0,1]$  and $\{\gamma_n\}_{n\ge 0},\{\mu_n\}_{n\ge 0}\in [\varepsilon,1/\varepsilon]$}
			\FOR{$n=0,1\dots$}
			\STATE{\bf{Fejerian step}}
			%\STATE{Choose $\secondhsubset_n$ such that $Z\subset \secondhsubset_n$}
			\FOR{$i=1,\dots,M$}
			\STATE{$a_{i,n} = \text{Prox}_{\gamma_n f_i} (p_{i,n} - \gamma_n\sum_{k=1}^{K} L_{ki}^* v_{k,n}^*)$}
			\STATE{$a_{i,n}^* = \gamma_n^{-1}(p_{i,n}-a_{i,n}) - \sum_{k=1}^{K} L_{ki}^* v_{k,n}^* $}
			\ENDFOR
			\FOR{$k=1,\dots,K$}
			\STATE{$b_{k,n} = \text{Prox}_{\mu_n g_k} (\sum_{i=1}^{M}L_{ik} p_{i,n} + \mu_n  v_{k,n}^*)$}
			\STATE{$b_{k,n}^* = \mu_n^{-1} (\sum_{i=1}^{M}  L_{ik} p_{i,n}-b_{k,n}) +  v_{k,n}^*  $}
			\STATE{$s_{M+k,n}^*=b_{k,n}-\sum_{i=1}^{M}L_{ik} a_{i,n} $}
			\ENDFOR
			\FOR{$i=1,\dots,M$}
			\STATE{ $s_{i,n}^*=a_{i,n}^*+ \sum_{k=1}^{K} L_{ki}^* b_{k,n}^* $}
			\ENDFOR
			\STATE{ $s_n^*= (s_{1,n}^*,\dots,s_{M,n}^*,s_{M+1,n}^*,\dots,s_{M+K,n}^*)$}%\prod_{j=1}^{M+K} s_{j,n}^* $}
			\STATE{$\eta_n=\sum_{i=1}^M \langle a_{i,n} \ |\ a_{i,n}^* \rangle + \sum_{k=1}^{K} \langle b_{k,n} \ |\ b_{k,n}^* \rangle $}
			\STATE{$H_n = \left\{h \in E \mid \left\langle  h \mid s_n^*\right\rangle \leq \eta_n  \right\}$}
			\IF{$\|s_n^*\|=0$}
			\STATE{$\bar{p}=p_n$, $\bar{v}^*=v_n^*$}
			\STATE{Terminate}
			\ELSE
			\STATE{$x_{n+1/2}=x_{n}+\lambda_n(P_{\secondhsubset_n}(x_{n})-x_{n})$}
			\STATE{\bf{Haugazeau step}}
			\STATE{Choose $\hsubset_n$  closed convex such that $Z\subset C_n \subset H(x_n,x_{n+1/2})$
				% according to Proposition \ref{prop:setdefinitions}
			} 
			\STATE{$x_{n+1}=P_{H(x_0, x_n) \cap \hsubset_n}(x_0)$}
			\ENDIF
			\ENDFOR
			\RETURN
		\end{algorithmic}
	\end{algorithm}

	Let us note that Proposition \ref{theorem:best_projection} can be easily generalized to cover also the case of the set $Z$ defined by \eqref{set:KTmulti}.

	\section{Experimental results}
	\label{sec:experiments}
	
	The goal of this section is to illustrate and analyze the performance of the proposed
	Iterative Scheme \ref{alg:proximal_best2}
	in solving problem \eqref{prob:1multi}, i.e. we aim at illustrating the main contribution of our work: (a) to show experimentally the influence of the choice of set $C_n$ on the convergence of the algorithm and (b) to show experimentally that the proposed attraction property provide an additional measure of the distance of the current iterate to the solution.
	We provide numerical results related to
	simple convex image inpaiting problem. 
	The considered problem can be rewritten as an instance of \eqref{prob:1multi} by setting $M=1$, $K=2$, $H = \mathbb{R}^{3D}$ and finding $\min_{p\in H} \quad  f_1(p) + \sum_{k=1}^{2} g_k\left( L_{k}p\right)$, where functions $f_1$, $g_1$ and $g_2$  correspond to positivity constraint, data fidelity term and total variation (TV) based regularization \cite{Rudin_1992_total_variation}, respectively. 
	We focus on the analysis of influence of the choice of $C_{n}$ on the convergence. To this end we report the number of iterations of the algorithm with different $C_{n}$ settings performed to reach a tolerance $\Vert p_{n+1}-p_{n}\Vert / \left( 1 + \Vert p_{n}\Vert\right) $ less than $\epsilon$ in two successive iterations. The considered algorithms are denoted hereafter by PDBA-C0, PDBA-C1, PDBA-C2, PDBA-C3, for  $C_{n}=H(x_{n},x_{n+1/2})$ and $C_{n}$ defined by \eqref{set:halfspace1}, \eqref{set:halfspace2} and \eqref{set:halfspace3}, respectively. In the case of  PDBA-C3 $\tau_n$ is set to 0.5.
	Numerically, the convergence rate improvement is measured by ItR defined as a ratio of
	the numbers of
	iterations consumed by PDBA-Ci (where i takes value 0,1,2,3) and those consumed by PDBA-C0.
	The  algorithms performance is illustrated by the following curves: (a) signal to noise ratio (SNR) and (b) the bounds given by Proposition \ref{prop:convergence_analysis}  as a functions of iteration number.
	%Finally we compare empirical convergence rate with the theoretical bounds given by Proposition \ref{prop:convergence_analysis} and Proposition \ref{prop:inertial_property}. 

	The evaluation experiments concern the
	image inpainting problem which corresponds to the recovery of an image $\bar{p} \in  \mathbb{R}^{3D}$ 
	from lossy observations 
	$y =L_1\bar{p}$, where $L_1 \in \mathbb{R}^{3D\times 3D}$ is a diagonal matrix  such that for $i = 1, \ldots , D$ we have $L_1 (i,i) = L_1 (2i,2i)=L_1 (3i,3i)=0$, if the pixel $i$ in the observation image $y$ is lost and  $L_1 (i,i)= L_1 (2i,2i)=L_1 (3i,3i) = 1$, otherwise.  
	The considered optimization problem is of the form
	\begin{equation}
		\min_{p \in H} \quad \iota_{y}(L_1 p)   + \iota_{\mathcal{S}}(p) + TV( p)
		\label{prob:inpainting}
	\end{equation}
	where 
	$\iota$ is the indicator function defined as:
	\begin{equation}
		\iota_\mathcal{S}(p) = 
		\begin{cases}
			0 & \; \; \; \text{if}\; \; \;  p \in \mathcal{S}\\
			+\infty & \text{otherwise},\\ 
		\end{cases}
		\label{eq:iden_func}
	\end{equation}
	$TV : \mathbb{R}^{3D} \mapsto \mathbb{R}$ is a discrete isotropic total
	variation functional~\cite{Rudin_1992_total_variation}, i.e. for every $p \in
	\mathbb{R}^{3D}$, $TV(p)=g(L_2 p): =\omega \left(\sum_{d=1}^D([\Delta^{\mathrm{h}} p]_d)^2 + ([\Delta^{\mathrm{v}}  p]_d)^2 \right)^{1/2}$
	with $L_2 \in \mathbb{R}^{6D\times 3D}, L_2:= \left[
	(\Delta^{\mathrm{h}})^\top\;\;(\Delta^{\mathrm{v}})^\top
	\right]^\top$,  where $\Delta^{\mathrm{h}}\in \mathbb{R}^{3D\times 3D}$
	(resp.  $\Delta^{\mathrm{v}} \in \mathbb{R}^{3D\times 3D}$) corresponds to a
	horizontal (resp. vertical) gradient operator,\\
	%$g_k(L_k p) =\omega \sum_{d=1}^D\left(([\Delta^{\mathrm{h}} p]_d)^2 + ([\Delta^{\mathrm{v}}  p]_d)^2 \right)^{1/2}$
	\begin{align*}
		&[\Delta^{\mathrm{h}} p]_d:=[(\Delta^{\mathrm{h}} p)_{d},(\Delta^{\mathrm{h}} p)_{2d},(\Delta^{\mathrm{h}} p)_{3d}]\in \mathbb{R}^{3},\\
		&[\Delta^{\mathrm{v}} p]_d:=[(\Delta^{\mathrm{v}} p)_{d},(\Delta^{\mathrm{v}} p)_{2d},(\Delta^{\mathrm{v}} p)_{3d}] \in \mathbb{R}^{3}
	\end{align*}
	and $\omega$ denotes regularization parameter .
	The function $\iota_{\mathcal{S}}(p)$ is imposing the solution to belong to the set $\mathcal{S} = \left[ 0,1\right]^{3D} $.
	The dual problem to \eqref{prob:inpainting} is  the following
	optimization problem \cite[Example 3.24, 3.26, 3.27]{boyd2004convex}:
	\begin{align}
		\begin{aligned}	\label{inpainting:dual}
			\min_{v_1 \in G_1,\ v_2 \in G_2}\quad  & \langle y\ |\ v_1 \rangle 
			+ \sup_{r \in S} \langle r\ |\ - L_1^* v_1 - L_2^*v_2 \rangle + 
			TV^*(v_2)		  
		\end{aligned}
	\end{align}
	where $TV^*(v_2) = \omega \iota_{B}
	(\frac{v_2}{\omega} )$, convex set $B=\lbrace v \in \mathbb{R}^{6D}: \Vert v \Vert_2 \leq 1 \rbrace$, $G_1=\mathbb{R}^{3D}$, $G_2=\mathbb{R}^{6D}$.
	%For the corresponding conjugate functions required for computation primal dual gap related to formulation \ref{prob:2multi} we refer to \cite[Example 3.24, 3.26, 3.27]{boyd2004convex}.
	In the following experiments, we consider the cases of lossy observations with $\kappa$  randomly chosen pixels which are unknown.

	%The number of iteration required by PDBA-C1 with respect to PDBA-C0 is 0.4-0.75 times smaller. 
	%The smaller number of iterations is required in case of $\gamma_n$, $\mu_n$ set around 1.5 (see Table \ref{tab:res_inpainting_e001_gamma1.5}), but the algorithms stop with lower value of SNR. For $\gamma_n$, $\mu_n$ set around $0.01$ the high number of iterations is required due to osculating character of convergence curve, 
	
	%convergence requires the slope of the supply curve to be less 
	
	%
	%\begin{table} \renewcommand{\arraystretch}{1.5} 
	%\begin{center}
	%\begin{tabular}{|c|c |c |c |c| c |c |c |c |}
	%\multicolumn{1}{c}{}&\multicolumn{2}{c}{PDBA-C0} &  \multicolumn{2}{c}{PDBA-C1} &  \multicolumn{2}{c}{PDBA-C2} & \multicolumn{2}{c}{PDBA-C3} \\ \cline{2-9}
	%\multicolumn{1}{c|}{}&It0 & SNR0 & $\frac{It1}{It0}$ & SNR1 & $\frac{It2}{It0}$ & SNR2 & $\frac{It3}{It0}$ & SNR3 \\ \hline
	%20$\%$ & 172& 19.18&0.78 &19.51 & 19.23&1 &19.23&1 \\ \hline
	%40$\%$ & 154& 17.93&0.54 & 17.61&1.04 &18.00& 1.04&18.00 \\ \hline
	%60$\%$ &151 &16.54 &0.63 &16.45 &0.92 & 16.54& 0.97&16.54 \\ \hline
	%80$\%$ & 136& 15.09 & 0.49& 14.96&0.92 &15.09 & 1& 15.09\\ \hline
	%90$\%$ & 147 &13.41 & 0.48&13.45 & 1& 13.41& 1& 13.41\\ \hline
	%\end{tabular}
	%\caption{Reconstruction results from incomplete coefficients with  $\epsilon = 0.05$,  $\gamma_n = 0.33$, $\mu_n = 0.33$. \label{tab:res_inpainting}}
	%\end{center}
	%\end{table}

	In the following we examine the cases of $\kappa$ set to 20$\%$, 40$\%$, 60$\%$, 80$\%$, 90$\%$ (hereafter $\tilde{\kappa}$ denotes a fraction of missing pixels).
	For all  the algorithms we used the initialization $x_0 = \left[ y, L_1 y , L_2 y\right]^T $. The test were performed on image \texttt{fruits} from public domain (source: \textit{www.hlevkin.com/TestImages}) 
	of size $D =240 \times 256$. 
	
	In our first experiment, we study the influence of the choice of $C_{n}$ for different settings of $\gamma_n$, $\mu_n$ and $\lambda_n$, which play a significant role in convergence analysis. 
	%We examine the settings where for all iterations $\gamma_n$ and $\mu_n$ are set to some constant.
	The results summarized in Tables \ref{tab:res_inpainting_e001_gamma005},\ref{tab:res_inpainting_e001_gammma001}, \ref{tab:res_inpainting_e001_gamma1.5},  correspond to the 
	choice of $\gamma_n=\mu_n$ equal to 0.005, 0.01 and 1.5, respectively.
	These results 
	show that independently of the choice of parameters $\gamma_n$, $\mu_n$ algorithm PDBA-C1 leads to the best performance, while the results obtained with PDBA-C2 and PDBA-C3 are comparable to PDBA-C0. 
	Specifically, within our setting the numbers of
	iterations consumed by PDBA-C1 range from $40\%$ to $75\%$ of those consumed by PDBA-C0, while the SNR, values of TV and inpainting residues are negligible. By inspecting Tables \ref{tab:res_inpainting_e001_gamma005},\ref{tab:res_inpainting_e001_gammma001}, \ref{tab:res_inpainting_e001_gamma1.5}, one can observe that the obtained results depend strongly upon to the choice of $\gamma_n$ and $\mu_n$. 
	
	We would like to emphasize that ideally the termination tolerance should be a function of parameters $\gamma_n$ , $\mu_n$ and $\lambda_n$.
	The results presented in Table \ref{tab:res_inpainting_e001_summary} shows that in the case of $\gamma_n$ and $\mu_n$ equal to $0.003$ or $100$
	the tolerance should be smaller to prevent premature termination. In these cases the iteration number is very low, however the values of TV and SNR are significantly different than for the other choices.
	The  premature termination is due to flat slope of the convergence curve.
	Similar effect can be observed when $\lambda_n = 0.8$ (see Table \ref{tab:res_inpainting_e001_lambda}). %For values of $\lambda_n$ close to 1, the number of iterations required for convergence decreases. However, starting from some value of $\lambda_n$ the convergence curve is becoming very flat 
	\newpage
	
	\begin{table} [H] \renewcommand{\arraystretch}{1.5}

		\begin{tabular}{c |c |c |c| c |}
			\multicolumn{1}{c}{\bf{$\kappa=$20$\%$}}& \multicolumn{1}{|c}{PDBA-C0} & \multicolumn{1}{c}{PDBA-C1} & \multicolumn{1}{c}{PDBA-C2}& \multicolumn{1}{c}{PDBA-C3} \\ \cline{1-5}
			%\multirow{4}{*}{PDBA-C0}
			% \parbox[t]{2mm}{\multirow{4}{*}{\rotatebox[origin=c]{90}{20$\%$}}} 
			ItR  &1 & \bf{0.40}& 1.02& 1.02\\ \cline{2-5}
			SNR&  24.19& 24.25& 24.18& 24.19\\ \cline{2-5}
			TV & 40.66 & 40.37& 40.68& 40.67\\ \cline{2-5}
			{$\frac{\Vert  y - L_1 p \Vert_{1}} {(1-\tilde{\kappa}) 3D}$} &
			0.004  &  0.004   & 0.004   & 0.004   \\ \cline{2-5}
			%0.0422    0.0428    0.0421    0.0422 - srednie odchylenia na niezerowy pikel
			\multicolumn{1}{c}{}&\multicolumn{1}{c}{} &\multicolumn{1}{c}{} &\multicolumn{1}{c}{} &\multicolumn{1}{c}{}\\
			
			%\multicolumn{1}{c}{} &  \multicolumn{4}{c}{20$\%$} \\ \cline{2-5}
			\multicolumn{1}{c}{\bf{ $\kappa=$ 40$\%$}}& \multicolumn{1}{|c}{PDBA-C0} & \multicolumn{1}{c}{PDBA-C1} & \multicolumn{1}{c}{PDBA-C2}& \multicolumn{1}{c}{PDBA-C3} \\ \cline{1-5}
			%\multirow{4}{*}{PDBA-C0}
			ItR &1 &\bf{0.51} &0.98& 1.02\\ \cline{2-5}
			SNR & 20.64&20.60 &20.64 & 20.64\\ \cline{2-5}
			TV &36.06 &36.04 & 36.05& 36.07\\ \cline{2-5}
			{$\frac{\Vert  y - L_1 p \Vert_{1}} {(1-\tilde{\kappa}) 3D}$} 
			&   0.002 &  0.003 &    0.002 &    0.002\\ \cline{2-5}
			%\multicolumn{5}{l}{Results with $\kappa=$ 20$\%$}
			
			\multicolumn{1}{c}{}&\multicolumn{1}{c}{} &\multicolumn{1}{c}{} &\multicolumn{1}{c}{} &\multicolumn{1}{c}{}\\
			%\multicolumn{1}{c}{} &  \multicolumn{4}{c}{20$\%$} \\ \cline{2-5}
			\multicolumn{1}{c}{\bf{ $\kappa=$ 60$\%$}}& \multicolumn{1}{|c}{PDBA-C0} & \multicolumn{1}{c}{PDBA-C1} & \multicolumn{1}{c}{PDBA-C2}& \multicolumn{1}{c}{PDBA-C3} \\ \cline{1-5}
			%\multirow{4}{*}{PDBA-C0}
			ItR &1 &  \bf{0.44} & 1.01& 0.36\\ \cline{2-5}
			SNR &  18.26& 18.28 & 18.26 & 17.90\\ \cline{2-5}
			TV & 30.87& 30.55& 30.87& 30.80\\ \cline{2-5}
			{$\frac{\Vert  y - L_1 p \Vert_{1}} {(1-\tilde{\kappa}) 3D}$} 
			& 0.002 &    0.002 &   0.002 & 0.002\\ \cline{2-5}
			
			%60$\%$ &7779 & 18.26&\bf{0.44} &18.28 & 1.01&18.26 &0.36 & 17.90\\ \hline
			\multicolumn{1}{c}{}&\multicolumn{1}{c}{} &\multicolumn{1}{c}{} &\multicolumn{1}{c}{} &\multicolumn{1}{c}{}\\
			
			%\multicolumn{1}{c}{} &  \multicolumn{4}{c}{20$\%$} \\ \cline{2-5}
			\multicolumn{1}{c}{\bf{ $\kappa=$ 80$\%$}}& \multicolumn{1}{|c}{PDBA-C0} & \multicolumn{1}{c}{PDBA-C1} & \multicolumn{1}{c}{PDBA-C2}& \multicolumn{1}{c}{PDBA-C3} \\ \cline{1-5}
			%\multirow{4}{*}{PDBA-C0}
			ItR &1 &\bf{0.49} & 0.89& 0.99\\ \cline{2-5}
			SNR & 16.17 & 16.18 &16.18 & 16.17\\ \cline{2-5}
			TV & 23.79& 23.50& 23.74& 23.79 \\ \cline{2-5}
			%{$\Vert  y - L_1 p \Vert_{\infty}$
			%} & 2.35 -2&   5.45 -2&   2.45 -2&  2.39 -2\\ \cline{2-5}
			{$\frac{\Vert  y - L_1 p \Vert_{1}} {(1-\tilde{\kappa}) 3D}$} &
			0.001 &    0.001 &    0.001 &    0.001
			\\ \cline{2-5}
			%\multicolumn{5}{l}{Results with $\kappa=$ 20$\%$}
			%80$\%$ &5084 &16.17 &0.49 &16.18 &0.89 & 16.18& 0.99& 16.17\\ \hline
			\multicolumn{1}{c}{}&\multicolumn{1}{c}{} &\multicolumn{1}{c}{} &\multicolumn{1}{c}{} &\multicolumn{1}{c}{}\\
			
			%\multicolumn{1}{c}{} &  \multicolumn{4}{c}{20$\%$} \\ \cline{2-5}
			\multicolumn{1}{c}{\bf{ $\kappa=$ 90$\%$}}& \multicolumn{1}{|c}{PDBA-C0} & \multicolumn{1}{c}{PDBA-C1} & \multicolumn{1}{c}{PDBA-C2}& \multicolumn{1}{c}{PDBA-C3} \\ \cline{1-5}
			%\multirow{4}{*}{PDBA-C0}
			ItR &1 &\bf{0.51} & 1.08& 0.99\\ \cline{2-5}
			SNR & 14.71& 14.62 & 14.70 & 14.71\\ \cline{2-5}
			TV &  18.87&   18.13&   18.94&   18.87\\ \cline{2-5}
			{$\frac{\Vert  y - L_1 p \Vert_{1}} {(1-\tilde{\kappa}) 3D}$} 
			&0.001 &   0.001 &    0.001 &    0.001\\ \cline{2-5}
			%\multicolumn{5}{l}{Results with $\kappa=$ 20$\%$} 
			%90$\%$ & 5508& 14.71 &0.51 & 14.62& 1.08& 14.70& 0.99&14.71 \\ \hline
		\end{tabular}
		\caption{Reconstruction results from incomplete data with  $\epsilon = 10^{-2}$,  $\lambda_n=1$, $\gamma_n = 0.005$, $\mu_n = 0.005$. \label{tab:res_inpainting_e001_gamma005}}
	\end{table}

	\begin{table} [H] \renewcommand{\arraystretch}{1.5}

		\begin{tabular}{c |c |c |c| c |}
			\multicolumn{1}{c}{\bf{$\kappa=$20$\%$}}& \multicolumn{1}{|c}{PDBA-C0} & \multicolumn{1}{c}{PDBA-C1} & \multicolumn{1}{c}{PDBA-C2}& \multicolumn{1}{c}{PDBA-C3} \\ \cline{1-5}
			%\multirow{4}{*}{PDBA-C0}
			% \parbox[t]{2mm}{\multirow{4}{*}{\rotatebox[origin=c]{90}{20$\%$}}} 
			ItR  &1 & \bf{0.44}& 1 & 1 \\ \cline{2-5}
			SNR&  24.02& 24.13& 24.02& 24.02\\ \cline{2-5}
			TV &   40.06 &  39.66 &   40.06 &   40.06 \\ \cline{2-5}
			{$\frac{\Vert  y - L_1 p \Vert_{1}} {(1-\tilde{\kappa}) 3D}$} 
			& 0.005 &  0.005 &   0.005 &   0.005\\ \cline{2-5}
			%20$\%$ & 6382 &24.02   & \bf{0.44}& 24.13  &1 &24.02 &  1&  24.02\\ \hline
			\multicolumn{1}{c}{}&\multicolumn{1}{c}{} &\multicolumn{1}{c}{} &\multicolumn{1}{c}{} &\multicolumn{1}{c}{}\\
			
			%\multicolumn{1}{c}{} &  \multicolumn{4}{c}{20$\%$} \\ \cline{2-5}
			\multicolumn{1}{c}{\bf{ $\kappa=$ 40$\%$}}& \multicolumn{1}{|c}{PDBA-C0} & \multicolumn{1}{c}{PDBA-C1} & \multicolumn{1}{c}{PDBA-C2}& \multicolumn{1}{c}{PDBA-C3} \\ \cline{1-5}
			%\multirow{4}{*}{PDBA-C0}
			ItR &1 &\bf{0.54} & 1.02& 1.03\\ \cline{2-5}
			SNR & 20.52&20.56 &20.51 & 20.53\\ \cline{2-5}
			TV &  35.62 & 35.34 &   35.67 &   35.61\\ \cline{2-5}
			{$\frac{\Vert  y - L_1 p \Vert_{1}} {(1-\tilde{\kappa}) 3D}$} 
			&  0.004 &    0.004 &    0.004 &   0.004\\ \cline{2-5}
			%\multicolumn{5}{l}{Results with $\kappa=$ 20$\%$}
			%40$\%$ & 6112 & 20.52  & \bf{0.54} & 20.56 &1.02 &20.51 & 1.03 & 20.53 \\ \hline

			\multicolumn{1}{c}{}&\multicolumn{1}{c}{} &\multicolumn{1}{c}{} &\multicolumn{1}{c}{} &\multicolumn{1}{c}{}\\
			%\multicolumn{1}{c}{} &  \multicolumn{4}{c}{20$\%$} \\ \cline{2-5}
			\multicolumn{1}{c}{\bf{ $\kappa=$ 60$\%$}}& \multicolumn{1}{|c}{PDBA-C0} & \multicolumn{1}{c}{PDBA-C1} & \multicolumn{1}{c}{PDBA-C2}& \multicolumn{1}{c}{PDBA-C3} \\ \cline{1-5}
			%\multirow{4}{*}{PDBA-C0}
			ItR &1 &  \bf{0.52} & 1.04& 0.91\\ \cline{2-5}
			SNR &  18.24& 18.25 & 18.22 & 18.23\\ \cline{2-5}
			TV &  30.19 &   29.97 &   30.37 &   30.17\\ \cline{2-5}
			{$\frac{\Vert  y - L_1 p \Vert_{1}} {(1-\tilde{\kappa}) 3D}$} 
			&  0.003 &    0.003  &  0.003 &    0.003\\ \cline{2-5}
			
			%60$\%$ & 6411 &  18.24 & \bf{0.52} & 18.25 &1.04 &18.22 & 0.91 &18.23  \\ \hline
			\multicolumn{1}{c}{}&\multicolumn{1}{c}{} &\multicolumn{1}{c}{} &\multicolumn{1}{c}{} &\multicolumn{1}{c}{}\\
			
			%\multicolumn{1}{c}{} &  \multicolumn{4}{c}{20$\%$} \\ \cline{2-5}
			\multicolumn{1}{c}{\bf{ $\kappa=$ 80$\%$}}& \multicolumn{1}{|c}{PDBA-C0} & \multicolumn{1}{c}{PDBA-C1} & \multicolumn{1}{c}{PDBA-C2}& \multicolumn{1}{c}{PDBA-C3} \\ \cline{1-5}
			%\multirow{4}{*}{PDBA-C0}
			ItR &1 &\bf{0.59} & 1.21& 1.02\\ \cline{2-5}
			SNR & 16.15 & 16.16 &16.15 & 16.15\\ \cline{2-5}
			TV &  23.29 &  22.92 &   23.48 &   23.3202
			\\ \cline{2-5}
			{$\frac{\Vert  y - L_1 p \Vert_{1}} {(1-\tilde{\kappa}) 3D}$} 
			&  0.002 &   0.002 &   0.002 &   0.002\\ \cline{2-5}
			%\multicolumn{5}{l}{Results with $\kappa=$ 20$\%$}
			%80$\%$ & 5176 & 16.15 & \bf{0.59} & 16.16 & 1.21&16.15 &1.02  &16.15  \\ \hline
			\multicolumn{1}{c}{}&\multicolumn{1}{c}{} &\multicolumn{1}{c}{} &\multicolumn{1}{c}{} &\multicolumn{1}{c}{}\\
			
			%\multicolumn{1}{c}{} &  \multicolumn{4}{c}{20$\%$} \\ \cline{2-5}
			\multicolumn{1}{c}{\bf{ $\kappa=$ 90$\%$}}& \multicolumn{1}{|c}{PDBA-C0} & \multicolumn{1}{c}{PDBA-C1} & \multicolumn{1}{c}{PDBA-C2}& \multicolumn{1}{c}{PDBA-C3} \\ \cline{1-5}
			%\multirow{4}{*}{PDBA-C0}
			ItR &1 &\bf{0.71} & 1.52& 0.59\\ \cline{2-5}
			SNR & 14.67& 14.65 & 14.65 & 14.33\\ \cline{2-5}
			TV &  18.11 &   17.87 &   18.62 &   16.01\\ \cline{2-5}
			{$\frac{\Vert  y - L_1 p \Vert_{1}} {(1-\tilde{\kappa}) 3 D}$} 
			&0.001 &   0.001 &   0.001 &   0.002\\ \cline{2-5}
			%\multicolumn{5}{l}{Results with $\kappa=$ 20$\%$} 
			%90$\%$ & 4022 & 14.67  & \bf{0.71} & 14.65 & 1.52& 14.65 & 0.59 & 14.33 \\ \hline
		\end{tabular}
		\caption{Reconstruction results from incomplete data with  $\epsilon = 10^{-2}$,   $\lambda_n=1$, $\gamma_n = 0.01$, $\mu_n = 0.01$. \label{tab:res_inpainting_e001_gammma001}}
	\end{table}

	\begin{table} [H] \renewcommand{\arraystretch}{1.5}

		\begin{tabular}{c |c |c |c| c |}
			\multicolumn{1}{c}{\bf{$\kappa=$20$\%$}}& \multicolumn{1}{|c}{PDBA-C0} & \multicolumn{1}{c}{PDBA-C1} & \multicolumn{1}{c}{PDBA-C2}& \multicolumn{1}{c}{PDBA-C3} \\ \cline{1-5}
			%\multirow{4}{*}{PDBA-C0}
			% \parbox[t]{2mm}{\multirow{4}{*}{\rotatebox[origin=c]{90}{20$\%$}}} 
			ItR  &1 & \bf{0.75}& 1 & 1 \\ \cline{2-5}
			SNR&  23.00& 23.03& 23.00& 23.00\\ \cline{2-5}
			TV &  34.91 & 34.99 &  34.91 &   34.91\\ \cline{2-5}
			{$\frac{\Vert  y - L_1 p \Vert_{1}} {(1-\tilde{\kappa}) 3D}$} &
			0.011  &  0.011 &  0.011 &  0.011\\ \cline{2-5}
			%20$\%$ & 2883& 23.00&\bf{0.75} &23.03 & 1&23.00 &1&23.00 \\ \hline
			\multicolumn{1}{c}{}&\multicolumn{1}{c}{} &\multicolumn{1}{c}{} &\multicolumn{1}{c}{} &\multicolumn{1}{c}{}\\
			
			%\multicolumn{1}{c}{} &  \multicolumn{4}{c}{20$\%$} \\ \cline{2-5}
			\multicolumn{1}{c}{\bf{ $\kappa=$ 40$\%$}}& \multicolumn{1}{|c}{PDBA-C0} & \multicolumn{1}{c}{PDBA-C1} & \multicolumn{1}{c}{PDBA-C2}& \multicolumn{1}{c}{PDBA-C3} \\ \cline{1-5}
			%\multirow{4}{*}{PDBA-C0}
			ItR &1 &\bf{0.70} & 1.02& 1.03\\ \cline{2-5}
			SNR & 20.01&20.00 &20.01 & 20.01\\ \cline{2-5}
			TV &  31.19 &  31.17 &   31.19 &   31.19\\ \cline{2-5}
			{$\frac{\Vert  y - L_1 p \Vert_{1}} {(1-\tilde{\kappa}) 3D}$} &
			0.008 &   0.009 &   0.008 &    0.008\\ \cline{2-5}
			%\multicolumn{5}{l}{Results with $\kappa=$ 20$\%$}
			%40$\%$ & 2835& 20.01&\bf{0.70} &20.00 & 1&20.01 &1&20.01 \\ \hline
			\multicolumn{1}{c}{}&\multicolumn{1}{c}{} &\multicolumn{1}{c}{} &\multicolumn{1}{c}{} &\multicolumn{1}{c}{}\\
			%\multicolumn{1}{c}{} &  \multicolumn{4}{c}{20$\%$} \\ \cline{2-5}
			\multicolumn{1}{c}{\bf{ $\kappa=$ 60$\%$}}& \multicolumn{1}{|c}{PDBA-C0} & \multicolumn{1}{c}{PDBA-C1} & \multicolumn{1}{c}{PDBA-C2}& \multicolumn{1}{c}{PDBA-C3} \\ \cline{1-5}
			%\multirow{4}{*}{PDBA-C0}
			ItR &1 &  \bf{0.69} & 0.99& 1.00\\ \cline{2-5}
			SNR &  17.90& 17.89 & 17.89 & 17.89\\ \cline{2-5}
			TV & 26.60 &   26.63 &   26.59 &  26.59\\ \cline{2-5}
			{$\frac{\Vert  y - L_1 p \Vert_{1}} {(1-\tilde{\kappa}) 3D}$} &
			0.006 &   0.006 &    0.006 &    0.006\\ \cline{2-5}
			%60$\%$ & 2807& 17.90&\bf{0.69} &17.89 & 0.99&17.89 &1&17.89 \\ \hline
			\multicolumn{1}{c}{}&\multicolumn{1}{c}{} &\multicolumn{1}{c}{} &\multicolumn{1}{c}{} &\multicolumn{1}{c}{}\\
			
			%\multicolumn{1}{c}{} &  \multicolumn{4}{c}{20$\%$} \\ \cline{2-5}
			\multicolumn{1}{c}{\bf{ $\kappa=$ 80$\%$}}& \multicolumn{1}{|c}{PDBA-C0} & \multicolumn{1}{c}{PDBA-C1} & \multicolumn{1}{c}{PDBA-C2}& \multicolumn{1}{c}{PDBA-C3} \\ \cline{1-5}
			%\multirow{4}{*}{PDBA-C0}
			ItR &1 &\bf{0.73} & 1 & 1.02\\ \cline{2-5}
			SNR & 15.87 & 15.87 & 15.87 & 15.87\\ \cline{2-5}
			TV & 20.46 &   20.52 &   20.48 &   20.49 \\ \cline{2-5}
			{$\frac{\Vert  y - L_1 p \Vert_{1}} {(1-\tilde{\kappa}) 3D}$} &
			0.004 &   0.004&   0.004&   0.004\\ \cline{2-5}
			%\multicolumn{5}{l}{Results with $\kappa=$ 20$\%$}
			%80$\%$ & 2431& 15.87&\bf{0.73} &15.87 & 1.04&15.87 &1.02&15.87 \\ \hline
			\multicolumn{1}{c}{}&\multicolumn{1}{c}{} &\multicolumn{1}{c}{} &\multicolumn{1}{c}{} &\multicolumn{1}{c}{}\\
			
			%\multicolumn{1}{c}{} &  \multicolumn{4}{c}{20$\%$} \\ \cline{2-5}
			\multicolumn{1}{c}{\bf{ $\kappa=$ 90$\%$}}& \multicolumn{1}{|c}{PDBA-C0} & \multicolumn{1}{c}{PDBA-C1} & \multicolumn{1}{c}{PDBA-C2}& \multicolumn{1}{c}{PDBA-C3} \\ \cline{1-5}
			%\multirow{4}{*}{PDBA-C0}
			ItR &1 &\bf{0.74} & 0.95& 1.02\\ \cline{2-5}
			SNR & 14.33& 14.35 & 14.31 & 14.33\\ \cline{2-5}
			TV & 15.99 &   16.24 &  15.79 &   16.01\\ \cline{2-5}
			{$\frac{\Vert  y - L_1 p \Vert_{1}} {(1-\tilde{\kappa}) 3D}$} &
			0.002 &    0.002&    0.002&    0.002\\ \cline{2-5}
			%\multicolumn{5}{l}{Results with $\kappa=$ 20$\%$} 
			%90$\%$ &2319 &14.33 &\bf{0.74} &14.35 & 0.95&14.31 & 1.02& 14.33\\ \hline
		\end{tabular}
		\caption{Reconstruction results from incomplete data with  $\epsilon = 10^{-2}$,  $\lambda_n=1$,  $\gamma_n = 1.5$, $\mu_n = 1.5$. \label{tab:res_inpainting_e001_gamma1.5}}
	\end{table}

	\begin{table} [H] \renewcommand{\arraystretch}{1.5} 
		\begin{center}
			\begin{tabular}{p{0.75cm} c|c |c |c |c| c |}
				\multirow{5}{*}{\rotatebox[origin=c]{90}{\bf
						{PDBA-C0
						}}}
						&\multicolumn{1}{c|}{$\gamma_n$, $\mu_n$}&\multicolumn{1}{c}{0.003} &  \multicolumn{1}{c}{0.005} &  \multicolumn{1}{c}{0.01} & \multicolumn{1}{c}{1.5} & \multicolumn{1}{c}{100}\\ \cline{2-7}
						&\multicolumn{1}{c|}{SNR}& 7.28& 24.19 & 24.02 & 23.00 & 7.64 \\ \cline{2-7}
						&\multicolumn{1}{c|}{TV}& 130.29 &40.66 & 40.06 & 34.91& 120.99\\ \cline{2-7}
						&\multicolumn{1}{c|}{$\frac{\Vert  y - L_1 p \Vert_{1}} {(1-\tilde{\kappa}) 3D}$}&   0.003& 0.004& 0.005&  0.011 &  0.008\\ \cline{2-7}
						&\multicolumn{1}{c|}{It No.}& 2&7543& 6382 & 2883 & 5032 \\ \cline{2-7}
						%\end{tabular}
						\multicolumn{1}{c}{}&\multicolumn{1}{c}{}&\multicolumn{1}{c}{}&\multicolumn{1}{c}{}&\multicolumn{1}{c}{}&\multicolumn{1}{c}{}&\multicolumn{1}{c}{} \\
						%\begin{tabular}{|c|c |c |c |c| c |}
						\multirow{5}{*}{\rotatebox[origin=c]{90}{\bf
								{PDBA-C1
								}}}&\multicolumn{1}{c|}{$\gamma_n$, $\mu_n$}&\multicolumn{1}{c}{0.003} &  \multicolumn{1}{c}{0.005} &  \multicolumn{1}{c}{0.01} & \multicolumn{1}{c}{1.5} & \multicolumn{1}{c}{100}\\ \cline{2-7}
								&\multicolumn{1}{c|}{SNR}& 7.28& 24.25 & 24.13 & 23.03 & 7.55 \\ \cline{2-7}
								&\multicolumn{1}{c|}{TV}& 130.29 & 40.37& 39.66 &34.99 & 122.89\\ \cline{2-7}
								&\multicolumn{1}{c|}{$\frac{\Vert  y - L_1 p \Vert_{1}} {(1-\tilde{\kappa}) 3D}$}& 0.003 &0.004 &  0.005&  0.011& 0.007\\ \cline{2-7}
								&\multicolumn{1}{c|}{ItR}& 1&  0.40 & 0.44 & 0.75 & 0.83 \\ \cline{2-7}
							\end{tabular}
							\caption{Reconstruction results from incomplete data with  $\epsilon = 10^{-2}$, $\kappa=20\%$. \label{tab:res_inpainting_e001_summary}}
						\end{center}
					\end{table}

					In the second experiment, we compare 
					%the non-inertial and inertial version of the Algorithm, i.e. 
					PDBA-C0 and PDBA-C1 (best over  Algorithms with memory according to the first experiment). We present reconstruction results (see Tab.~\ref{tab:res_inpainting_e001_summary}) as well as supplying convergence curves (see Fig~\ref{fig:inpainting}), i.e.  SNR and bound as a function of iterations. Hereafter we call bounds as
					$-\|(p_0,v_0^*)-(p_n,v_{n}^*)\|^2=-\|x_0-x_n\|^2$ (see \eqref{prop:inertial_property:item2} of Proposition \ref{prop:inertial_property}).
					One can observe that PDBA-C1 leads to a faster convergence and the bounds are more tight (in the sense of Proposition \ref{prop:inertial_property}  \eqref{prop:inertial_property:item2}).
					The difference is the most important in the early stage of the iterations. Both algorithms slow down afterwards. For $\gamma_n=0.01$ (resp. $\mu_n = 0.01$) both versions of the algorithm lead to some numerical oscillations in convergence, which are no more visible for 
					settings  $\gamma_n=0.003$ (resp. $\mu_n = 0.003$).
					%provide an experiment, where we compare the
					%effect
					%Figure \ref{fig:inpainting}.. shows ... 
					\begin{table}[H] \renewcommand{\arraystretch}{1.5} 
						\begin{center}
							\begin{tabular}{|c|c |c |c |c| c |c |c |c }
								\multicolumn{1}{c}{}&\multicolumn{4}{c}{PDBA-C0} &  \multicolumn{4}{c}{PDBA-C1} \\ \cline{2-8}
								\multicolumn{1}{c|}{$\lambda_n$}& SNR0 & TV & $\frac{\Vert  y - L_1 p \Vert_{1}} {(1-\tilde{\kappa}) 3D}$ & It0 & SNR1  
								& TV & $\frac{\Vert  y - L_1 p \Vert_{1}} {(1-\tilde{\kappa}) 3D}$ & $\frac{It1}{It0}$\\ \cline{1-9}
								1 &  24.19& 40.66 & 0.004 & 7543 &   24.25 & 40.37& 0.004& \bf{0.40}\\ \hline
								0.95 & 24.32 & 40.36 & 0.003 & 5149 &  24.25  & 40.33& 0.004 &  \bf{0.56}\\ \hline
								0.9 & 24.24& 40.37 & 0.004 & 4584 & 24.26   & 40.26& 0.004&  \bf{0.57}\\ \hline
								0.8 & 7.30& 129.24 &  0.005 &  2 &   7.30 & 129.24&  0.005&  \bf{1}\\ \hline
							\end{tabular}
							\caption{Reconstruction results from incomplete coefficients with  $\epsilon = 10^{-2}$, $\kappa=20\%$, $\gamma_n = 0.005$, $\mu_n = 0.005$. \label{tab:res_inpainting_e001_lambda}}
						\end{center}
					\end{table}
					
		\begin{figure}[H]
			\centering
					\centering
					\subfigure[Original]{
						\centering
						\includegraphics[width=0.3\textwidth]{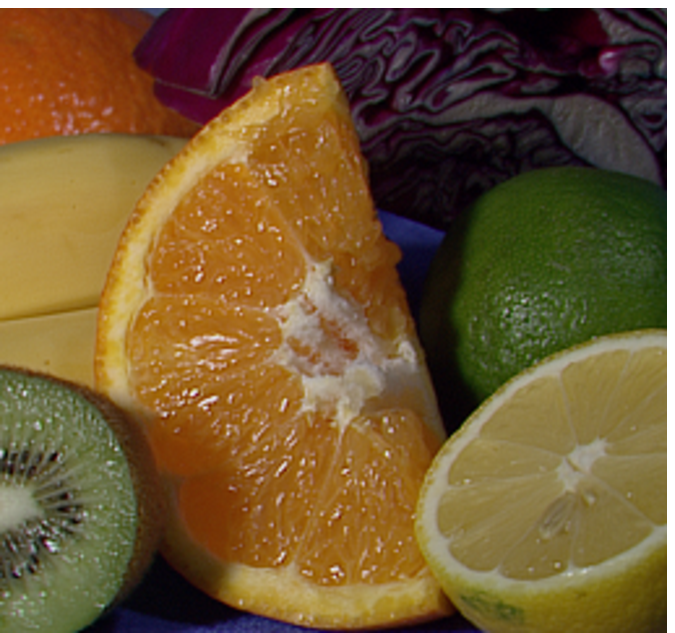}
					}%
					~ 
					\subfigure[Degraded]{
						\centering
						\includegraphics[width=0.3\textwidth]{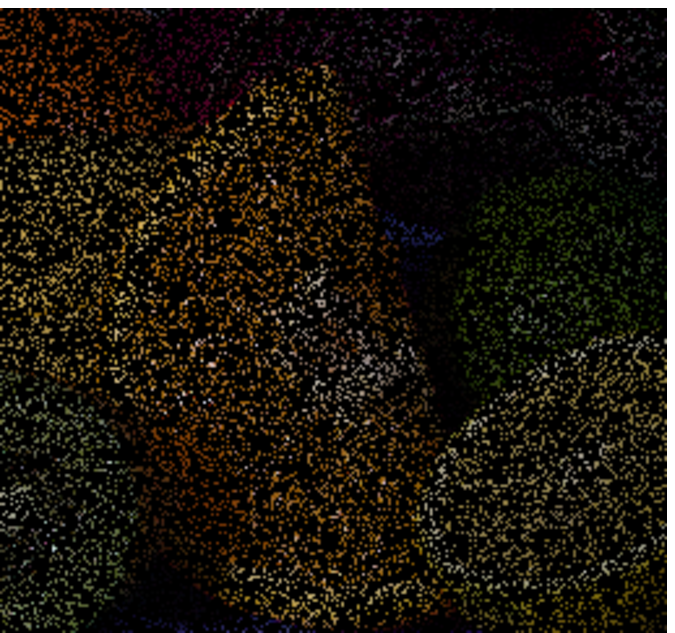}
					}
					~ 
					\subfigure[Ours result]{
						\centering
						\includegraphics[width=0.3\textwidth]{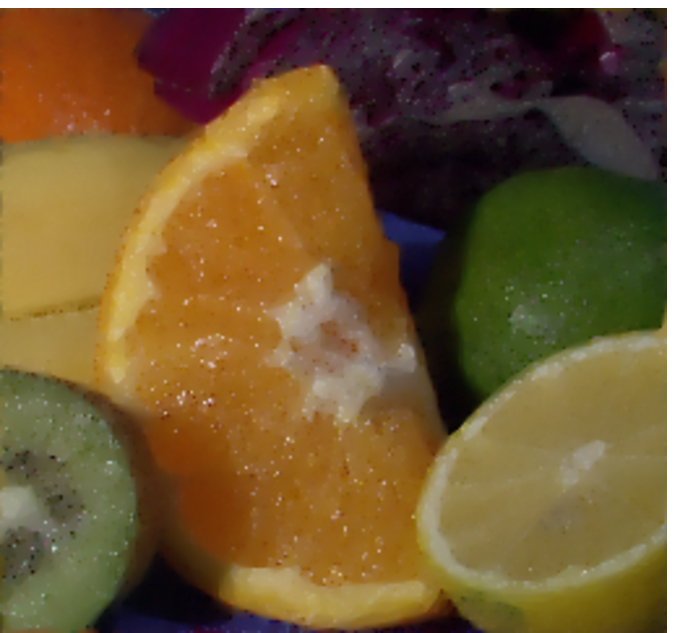}
					}
					~
					\subfigure[SNR ($\gamma_n = 0.01$, $\mu_n = 0.01$)]{
						\centering
						\includegraphics[width=0.4\textwidth]{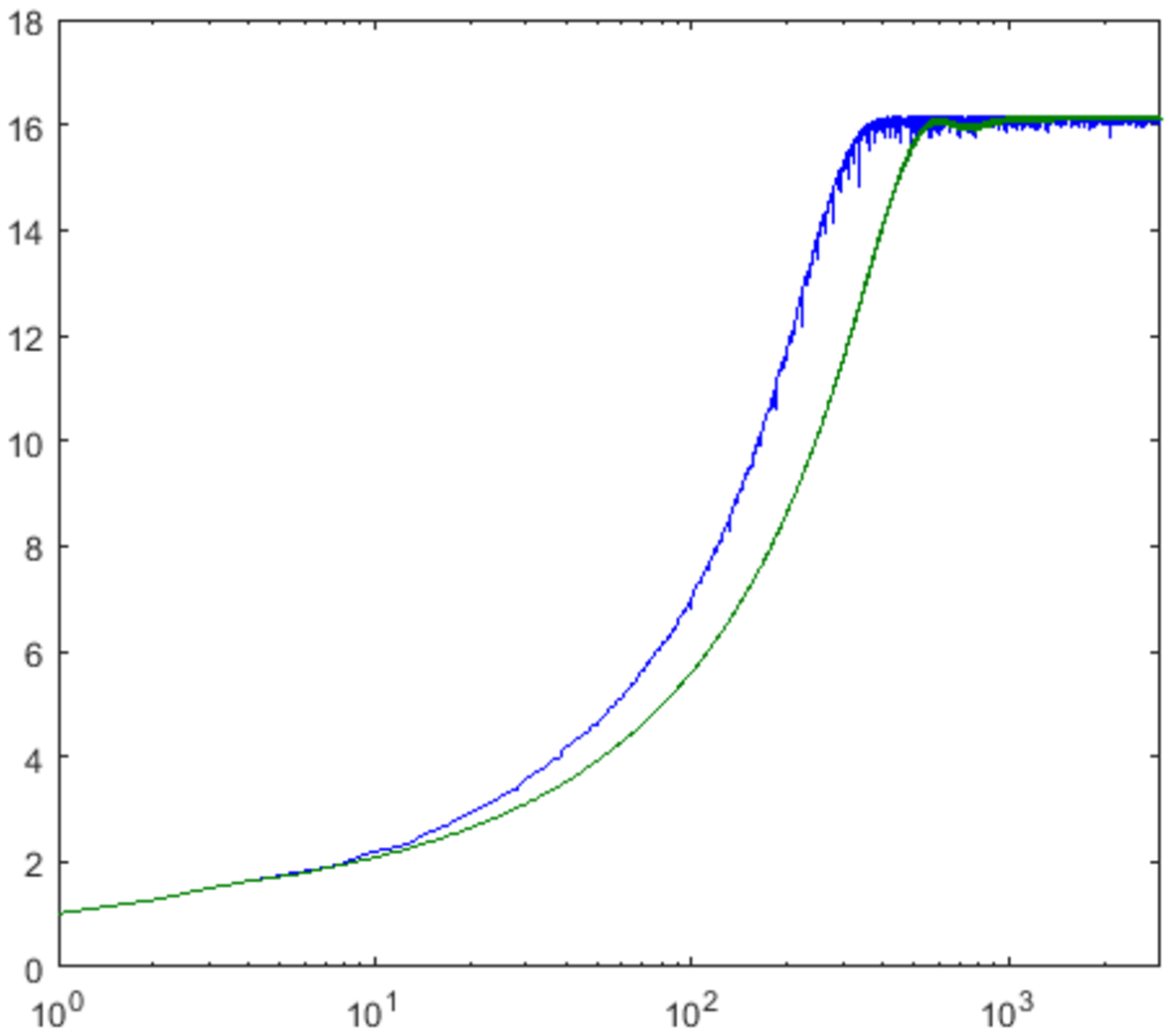}
					}%
					~ 
					\subfigure[Bounds ($\gamma_n = 0.01$, $\mu_n = 0.01$)]{
						\centering
						\includegraphics[width=0.4\textwidth]{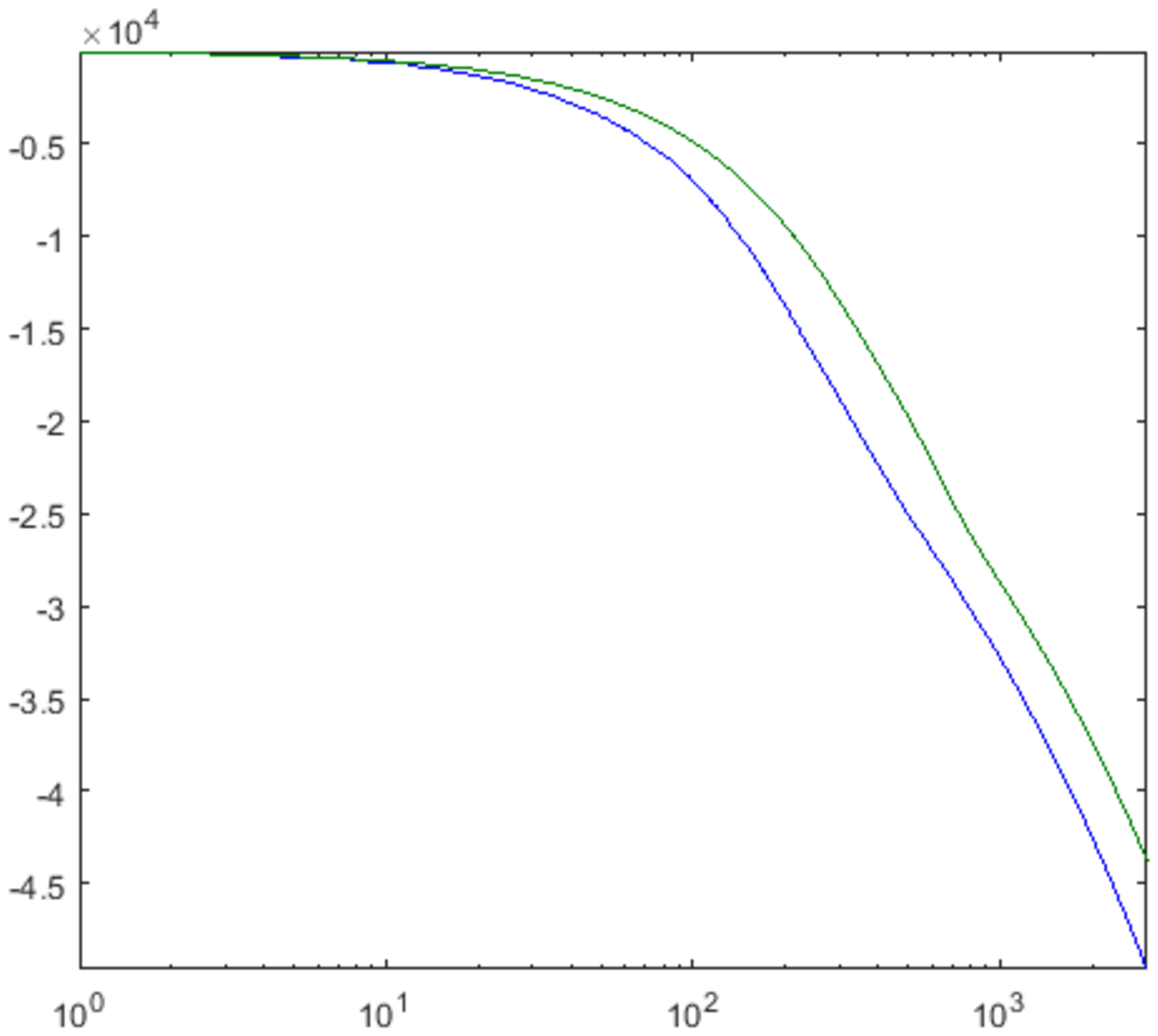}
						
					}
					\\ 
					\subfigure[SNR ($\gamma_n = 0.003$, $\mu_n = 0.003$)]{
						\centering
						\includegraphics[width=0.4\textwidth]{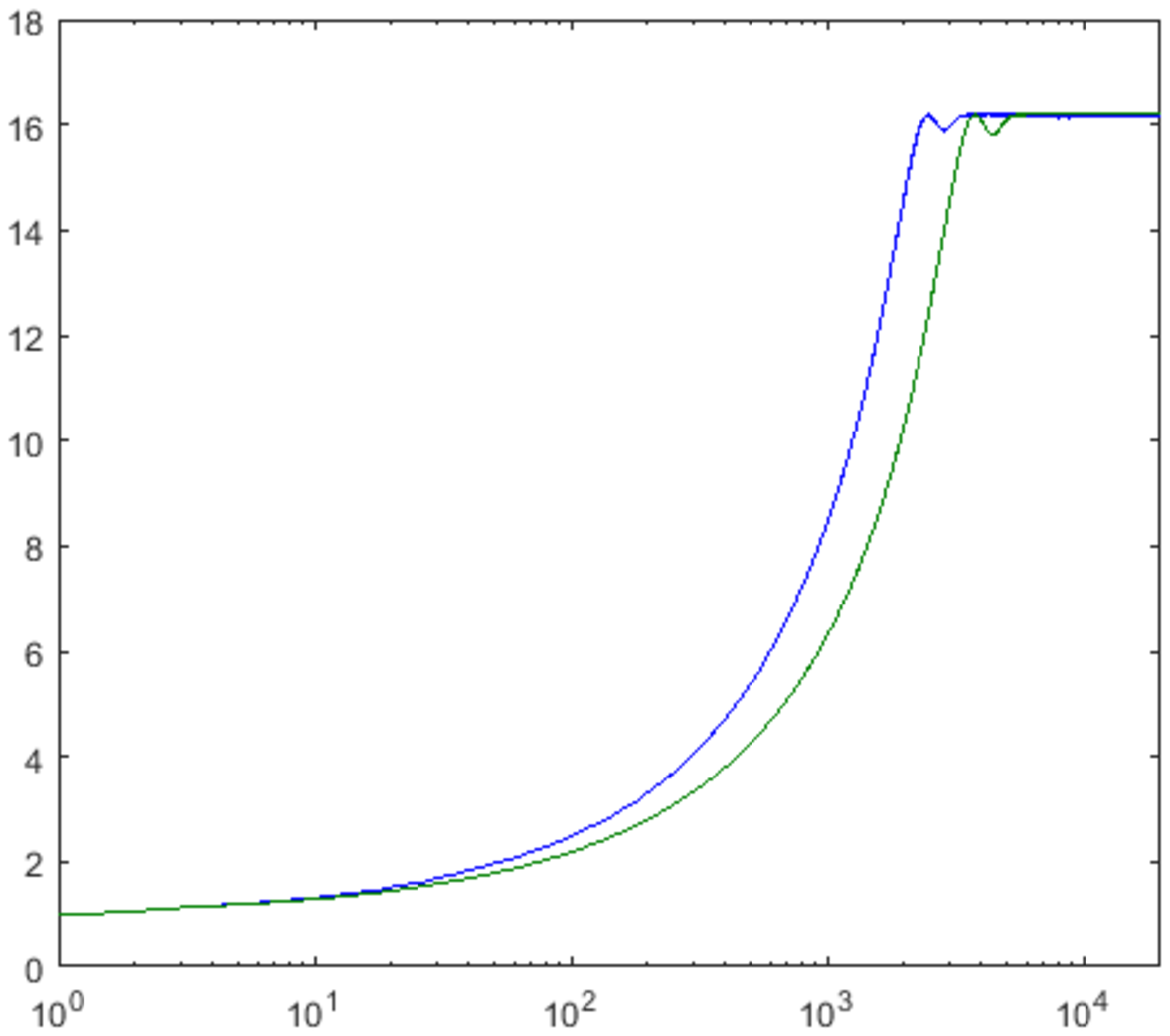}
					}
					~
					\subfigure[Bounds ($\gamma_n = 0.003$, $\mu_n = 0.003$)]{
						\centering
						\includegraphics[width=0.4\textwidth]{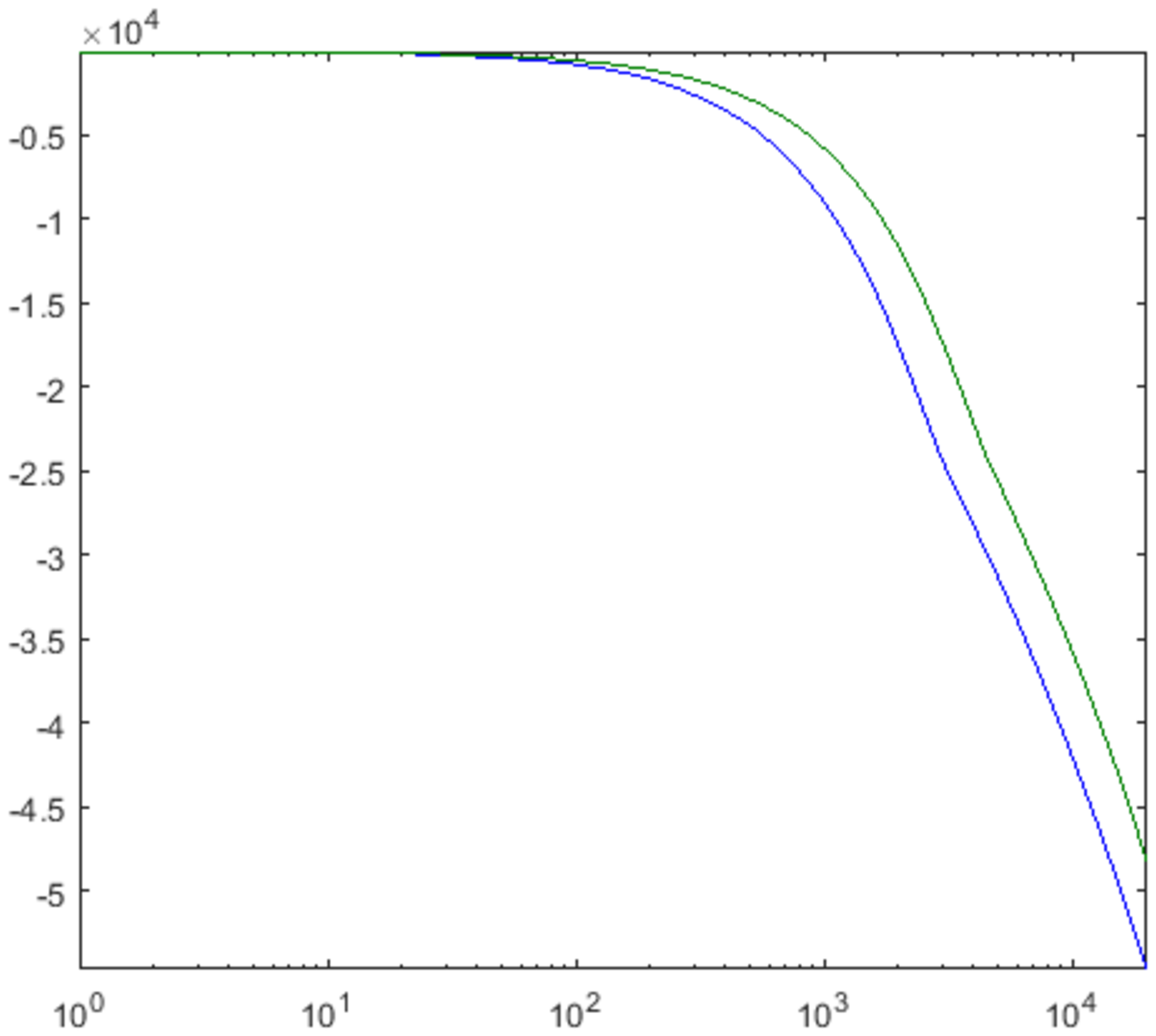}
					}%
										
			\caption{Figure (a) shows the $240 \times 256$ clean fruits image, (b) shows the same image for
				which 80$\%$ randomly chosen missing pixels, and (c) shows the solution generated
				by Algorithm \ref{alg:proximal_best2} PDBA-C1 after 3000 iterations, (d-e) and (f-g) show %primal-dual gap,
				SNR and \textit{attraction property} values (i.e $-\|(p_0,v_0^*)-(p_n,v_{n}^*)\|$) versus iterations, respectively. Algorithm  PDBA-C0 and PDBA-C1 are denoted in green and blue, respectively. \label{fig:inpainting}}
		\end{figure}

					\section{Conclusions}
					\label{sec:conclusions}
					In this paper we concentrate on a design of the novel scheme by incorporating  memory into projection algorithm.
					% We do not follow any of the previously presented paths of introducing inertial effect.
					We propose a new way of introducing memory effect into projection algorithms through incorporating into the algorithm projections onto  polyhedral sets built as intersections of halfspaces constructed with the help of current and previous iterates. To this end we provide the closed-form expressions and the algorithm  for  finding projections onto intersections of three halfspaces. Moreover, we adapt the general scheme proposed  in \cite{explicit_formulas_KR} to  particular halfspaces which may arise in the course of our Iterative Scheme. This allows us to limit the number of steps for finding projections. Building upon these results, we propose a new  primal-dual splitting algorithm with memory for solving convex optimization problems  via  general
					class of monotone inclusion problems involving parallel sum of maximally monotone
					operators composed with linear operators. 
					To analyse  convergence we prove the \textit{attraction property}. The \textit{attraction property} provides us with an evaluation criterion allowing to compare projection algorithms with and without memory. 
					Our experimental
					results related to preliminary implementation of the algorithms have shown that the proposed algorithm with memory
					%is
					generally 
					%faster
					needs smaller number of iterations 
					than the corresponding original one \cite{Alotaibi_2015_best}. Although only three strategies of introducing memory effect are analysed in this work, the generality of the presented theoretical results allow us to address versatility of the approach by constructing various forms of the algorithm which use information from former steps.

					\bibliographystyle{plain}
					\bibliography{references}

\begin{thebibliography}{10}

\bibitem{Alotaibi_2014_sucessive_fejer_approximation}
A.~Alotaibi, P.~L. Combettes, and N.~Shahzad.
\newblock Solving coupled composite monotone inclusions by successive
  {F}ej{\'e}r approximations of their {K}uhn-{T}ucker set.
\newblock {\em SIAM Journal on Optimization}, 24(4):2076--2095, 2014.

\bibitem{Alotaibi_2015_best}
A.~Alotaibi, P.~L. Combettes, and N.~Shahzad.
\newblock Best approximation from the {K}uhn-{T}ucker set of composite monotone
  inclusions.
\newblock {\em Numerical Functional Analysis and Optimization},
  36(12):1513--1532, 2015.

\bibitem{weak_convergence_of_relaxed_and_inertial_hybrid}
F.~Alvarez.
\newblock Weak convergence of a relaxed and inertial hybrid projection-proximal
  point algorithm for maximal monotone operators in {H}ilbert space.
\newblock {\em SIAM J. Optim.}, 14(3):773--782 (electronic), 2004.

\bibitem{Alvarez2001}
F.~Alvarez and H.~Attouch.
\newblock An inertial proximal method for maximal monotone operators via
  discretization of a nonlinear oscillator with damping.
\newblock {\em Set-Valued Analysis}, 9(1):3--11, 2001.

\bibitem{a_parallel_splitting_method}
H.~Attouch, L.~M. Briceno-Arias, and P.~L. Combettes.
\newblock A parallel splitting method for coupled monotone inclusions.
\newblock {\em SIAM J. Control Optim.}, 48(5):3246--3270, 2010.

\bibitem{A_Note_on_the_Paper_by_Eckstein_and_Svaiter}
H.~H. Bauschke.
\newblock A note on the paper by {E}ckstein and {S}vaiter on "general
  projective splitting methods for sums of maximal monotone operators".
\newblock {\em SIAM J. Control Optim.}, 48(4):2513--2515, 2009.

\bibitem{Bauschke_2011_convex_analysis}
H.~H. Bauschke and P.~L. Combettes.
\newblock {\em Convex Analysis and Monotone Operator Theory in {Hilbert} Spaces
  (CMS Books in Mathematics)}.
\newblock Springer, 2011.

\bibitem{an_inertial_alternating_direction}
R.~I. Bo{\c{t}} and E.~R. Csetnek.
\newblock An inertial alternating direction method of multipliers.
\newblock {\em Minimax Theory Appl.}, 1(1):29--49, 2016.

\bibitem{inertial_radu}
R.~I. Bo{\c{t}} and E.~R. Csetnek.
\newblock An inertial forward-backward-forward primal-dual splitting algorithm
  for solving monotone inclusion problems.
\newblock {\em Numerical Algorithms}, 71(3):519--540, 2016.

\bibitem{Bot_2015_convergence_rate_primal_dual}
R.~I. Bo{\c{t}}, E.~R. Csetnek, A.~Heinrich, and C.~Hendrich.
\newblock On the convergence rate improvement of a primal-dual splitting
  algorithm for solving monotone inclusion problems.
\newblock {\em Math. Program.}, 150(2):251--279, 2015.

\bibitem{inertial_douglas_rachford}
R.~I. Bo{\c{t}}, E.~R. Csetnek, and C.~Hendrich.
\newblock Inertial {D}ouglas-{R}achford splitting for monotone inclusion
  problems.
\newblock {\em Appl. Math. Comput.}, 256:472--487, 2015.

\bibitem{Bot2013}
R.~I. Bo{\c{t}} and C.~Hendrich.
\newblock A double smoothing technique for solving unconstrained
  nondifferentiable convex optimization problems.
\newblock {\em Computational Optimization and Applications}, 54(2):239--262,
  2013.

\bibitem{Douglas_Rachford_type}
R.~I. Bo{\c{t}} and C.~Hendrich.
\newblock A {D}ouglas-{R}achford type primal-dual method for solving inclusions
  with mixtures of composite and parallel-sum type monotone operators.
\newblock {\em SIAM J. Optim.}, 23(4):2541--2565, 2013.

\bibitem{boyd2004convex}
S.~Boyd and L.~Vandenberghe.
\newblock {\em Convex Optimization}.
\newblock Cambridge University Press, New York, NY, USA, 2004.

\bibitem{on_the_ergdoic_covergence_rates_Chambolle}
A.~Chambolle and T.~Pock.
\newblock On the ergodic convergence rates of a first-order primal--dual
  algorithm.
\newblock {\em Mathematical Programming}, pages 1--35, 2015.

\bibitem{inertial_proximal_admm}
C.~Chen, R.~H. Chan, S.~Ma, and J.~Yang.
\newblock Inertial proximal {ADMM} for linearly constrained separable convex
  optimization.
\newblock {\em SIAM J. Imaging Sci.}, 8(4):2239--2267, 2015.

\bibitem{a_general_inertial_proximal_point_algorithm_for_mixed_variational}
C.~Chen, S.~Ma, and J.~Yang.
\newblock A general inertial proximal point algorithm for mixed variational
  inequality problem.
\newblock {\em SIAM J. Optim.}, 25(4):2120--2142, 2015.

\bibitem{Chouzenoux_2016_block_coordinate_fb}
E.~Chouzenoux, J.-C. Pesquet, and A.~Repetti.
\newblock A block coordinate variable metric forward-backward algorithm.
\newblock {\em J. Global Optimization}, 66(3):457--485, 2016.

\bibitem{strong_convergence_fo_block}
P.~L. Combettes.
\newblock Strong convergence of block-iterative outer approximation methods for
  convex optimization.
\newblock {\em SIAM Journal on Control and Optimization}, 38(2):538--565, 2000.

\bibitem{Combettes_2009_fejer_monotonicity}
P.~L. Combettes.
\newblock {\em Encyclopedia of Optimization}, chapter Fej{\'e}r monotonicity in
  convex optimization, pages 1016--1024.
\newblock Springer US, Boston, MA, 2009.

\bibitem{asynchronous_block}
P.~L. Combettes and J.~Eckstein.
\newblock Asynchronous block-iterative primal-dual decomposition methods for
  monotone inclusions.
\newblock {\em Mathematical Programming}, pages 1--28, 2016.

\bibitem{Combettes_2015_best_bergman}
P.~L. Combettes and Q.~V. Nguyen.
\newblock Solving composite monotone inclusions in reflexive {B}anach spaces by
  constructing best {B}regman approximations from their {Kuhn-Tucker} set.
\newblock arXiv, 2015.

\bibitem{Combettes2012}
P.~L. Combettes and J.-C. Pesquet.
\newblock {Primal-Dual Splitting Algorithm for Solving Inclusions with Mixtures
  of Composite, Lipschitzian, and Parallel-Sum Type Monotone Operators}.
\newblock {\em Set-Valued and Variational Analysis}, 20(2):307--330, 2012.

\bibitem{a_simplified_form_of_block}
J.~Eckstein.
\newblock A simplified form of block-iterative operator splitting and an
  asynchronous algorithm resembling the multi-block alternating direction
  method of multipliers.
\newblock {\em Journal of Optimization Theory and Applications},
  173(1):155--182, 2017.

\bibitem{A_family_of_projective_splitting_methods}
J.~Eckstein and B.~F. Svaiter.
\newblock A family of projective splitting methods for the sum of two maximal
  monotone operators.
\newblock {\em Mathematical Programming}, 111(1):173--199, 2008.

\bibitem{general_projective_splitting_methods}
J.~Eckstein and B.~F. Svaiter.
\newblock General projective splitting methods for sums of maximal monotone
  operators.
\newblock {\em SIAM J. Control Optim.}, 48(2):787--811, 2009.

\bibitem{Haugazeau}
Y.~Haugazeau.
\newblock {\em Sur les Inequations Variationnelles et la Minimisation de
  Fonctionnelles Convexes}.
\newblock PhD thesis, Universite de Paris, 1968.

\bibitem{He2015}
N.~He, A.~Juditsky, and A.~Nemirovski.
\newblock Mirror prox algorithm for multi-term composite minimization and
  semi-separable problems.
\newblock {\em Computational Optimization and Applications}, 61(2):275--319,
  2015.

\bibitem{Hendrich_2014_phd}
C.~Hendrich.
\newblock {\em Proximal Splitting Methods in Nonsmooth Convex Optimization}.
\newblock PhD thesis, technische Universitat Chemnitz, 2014.

\bibitem{Johnstone2017}
P.~R. Johnstone and P.~Moulin.
\newblock Local and global convergence of a general inertial proximal splitting
  scheme for minimizing composite functions.
\newblock {\em Computational Optimization and Applications}, 67(2):259--292,
  2017.

\bibitem{numerical_approach_to_a_stationary_solution}
F.~Jules and P.~E. Maing{\'e}.
\newblock Numerical approach to a stationary solution of a second order
  dissipative dynamical system.
\newblock {\em Optimization}, 51(2):235--255, 2002.

\bibitem{playing_with_duality}
N.~Komodakis and J.-C. Pesquet.
\newblock Playing with duality: An overview of recent primal-dual approaches
  for solving large-scale optimization problems.
\newblock {\em IEEE Signal Processing Magazine}, 32(6):31--54, Nov 2015.

\bibitem{Li2016}
J.~Li, G.~Chen, Z.~Dong, and Z.~Wu.
\newblock A fast dual proximal-gradient method for separable convex
  optimization with linear coupled constraints.
\newblock {\em Computational Optimization and Applications}, 64(3):671--697,
  2016.

\bibitem{an_inertial_forward_backward_monotone_inclusions}
D.~A. Lorenz and T.~Pock.
\newblock An inertial forward-backward algorithm for monotone inclusions.
\newblock {\em Journal of Mathematical Imaging and Vision}, 51(2):311--325,
  2015.

\bibitem{inertial_iteratice_process_for_fixed_points}
P.-E. Maing{\'e}.
\newblock Inertial iterative process for fixed points of certain
  quasi-nonexpansive mappings.
\newblock {\em Set-Valued Anal.}, 15(1):67--79, 2007.

\bibitem{Mainge_2014_inertial_km_alg}
P.-E. Mainge.
\newblock Convergence theorems for inertial {KM-type} algorithms.
\newblock {\em Journal of Computational and Applied Mathematics}, 219(1):223 --
  236, 2008.

\bibitem{regularized_and_inertial_alg_for_common_fixed_points_nonlinear}
P.-E. Maing{\'e}.
\newblock Regularized and inertial algorithms for common fixed points of
  nonlinear operators.
\newblock {\em J. Math. Anal. Appl.}, 344(2):876--887, 2008.

\bibitem{a_hybrid_inertial_projection_proximal}
A.~Moudafi.
\newblock A hybrid inertial projection-proximal method for variational
  inequalities.
\newblock {\em JIPAM. J. Inequal. Pure Appl. Math.}, 5(3):Article 63, 5, 2004.

\bibitem{an_approximate_inertial_proximal_method_enlargement}
A.~Moudafi and E.~Elisabeth.
\newblock An approximate inertial proximal method using the enlargement of a
  maximal monotone operator.
\newblock {\em Int. J. Pure Appl. Math.}, 5(3):283--299, 2003.

\bibitem{Moudafi_2003_proximalconvergence}
A.~Moudafi and M.~Oliny.
\newblock Convergence of a splitting inertial proximal method for monotone
  operators.
\newblock {\em Journal of Computational and Applied Mathematics}, 155(2):447 --
  454, 2003.

\bibitem{Ochs_2015_ipiasco}
P.~Ochs, T.~Brox, and T.~Pock.
\newblock {iPiasco: Inertial Proximal Algorithm for Strongly Convex
  Optimization.}
\newblock {\em Journal of Mathematical Imaging and Vision}, 53(2):171--181,
  2015.

\bibitem{Ochs_2014_ipiano}
P.~Ochs, Y.~Chen, T.~Brox, and T.~Pock.
\newblock {iPiano: Inertial Proximal Algorithm for Nonconvex Optimization.}
\newblock {\em SIAM J. Imaging Sci.}, 7(2):1388--1419, 2014.

\bibitem{ARock}
Z.~Peng, Y.~Xu, M.~Yan, and W.~Yin.
\newblock Arock: An algorithmic framework for asynchronous parallel coordinate
  updates.
\newblock {\em SIAM Journal on Scientific Computing}, 38(5):A2851--A2879, 2016.

\bibitem{pennanen1999dualization}
T.~Pennanen.
\newblock {\em Dualization of monotone generalized equations}.
\newblock PhD thesis, University of Washington, 1999.

\bibitem{pennanen_dualization}
T.~Pennanen.
\newblock Dualization of generalized equations of maximal monotone type.
\newblock {\em SIAM J. Optim.}, 10(3):809--835, 2000.

\bibitem{a_parallel_inertial_proximal_optimization_method}
J.-C. Pesquet and N.~Pustelnik.
\newblock A parallel inertial proximal optimization method.
\newblock {\em Pac. J. Optim.}, 8(2):273--306, 2012.

\bibitem{Rockafellar_1970_convex_analysis}
R.~T. Rockafellar.
\newblock {\em Convex analysis}, volume~28.
\newblock Princeton University Press, Princeton, N.J, 1970.

\bibitem{a_stochastic_inertial_forward_backward}
L.~Rosasco, S.~Villa, and B.C. {V\~{u}}.
\newblock A stochastic inertial forward backward splitting algorithm for
  multivariate monotone inclusions.
\newblock {\em Optimization}, 65(6):1293--1314, 2016.

\bibitem{Rudin_1992_total_variation}
L~I Rudin, S~Osher, and E~Fatemi.
\newblock {Nonlinear total variation based noise removal algorithms}.
\newblock {\em Physica D: Nonlinear Phenomena}, 60(1-4):259--268, 1992.

\bibitem{explicit_formulas_KR}
Krzysztof~E. Rutkowski.
\newblock Closed-form expressions for projectors onto polyhedral sets in
  hilbert spaces.
\newblock {\em SIAM Journal on Optimization}, 27(3):1758--1771, 2017.

\bibitem{forcing_strong_2000}
M.V. Solodov and B.F. Svaiter.
\newblock {Forcing strong convergence of proximal point iterations in a Hilbert
  space}.
\newblock {\em Mathematical Programming}, 87(1):189--202.

\bibitem{strong_convergence_of_a_splitting_two}
G.~Tang and N.~Huang.
\newblock Strong convergence of a splitting proximal projection method for the
  sum of two maximal monotone operators.
\newblock {\em Operations Research Letters}, 40(5):332 -- 336, 2012.

\bibitem{strong_convergence_of_a_splitting_projection}
G.~Tang and F.~Xia.
\newblock Strong convergence of a splitting projection method for the sum of
  maximal monotone operators.
\newblock {\em Optimization Letters}, 8(4):1313--1324, 2014.

\bibitem{a_new_splitting}
Q.~{Tran-Dinh} and B.~{Cong Vu}.
\newblock {A new splitting method for solving composite monotone inclusions
  involving parallel-sum operators}.
\newblock {\em ArXiv e-prints}, May 2015.

\bibitem{VanHieu2017}
D.~Van~Hieu, P.~K. Anh, and L.~D. Muu.
\newblock Modified hybrid projection methods for finding common solutions to
  variational inequality problems.
\newblock {\em Computational Optimization and Applications}, 66(1):75--96,
  2017.

\bibitem{modified_extragradient_method}
Y.~J. Wang, N.~H. Xiu, and J.~Z. Zhang.
\newblock Modified extragradient method for variational inequalities and
  verification of solution existence.
\newblock {\em Journal of Optimization Theory and Applications},
  119(1):167--183, 2003.

\end{thebibliography}

					%%%%%%%%%%%%%%%%%%%%%%%%%%%%%%%%%%%%%%%%%%%%%%%%%%%%%%%%%%%%%%%%%%%%%%%%%%%%%%%%

\end{document}